\documentclass[12pt]{amsart}
\usepackage{amsmath}
\usepackage{amssymb}
\usepackage{amsfonts}
\usepackage{amsthm}
\usepackage{mathrsfs}
\usepackage{comment}
\usepackage[normalem]{ulem}
\usepackage{enumerate}
\usepackage{geometry}
\usepackage{mathtools}
\mathtoolsset{showonlyrefs}
\usepackage{hyperref}
\hypersetup{colorlinks=true,
linktoc=all,linkcolor=black,
citecolor=black}
\usepackage{tikz}
\usepackage{bbm}
\usepackage{mathptmx}
\newcommand{\eqd}{:=}
\newcommand{\id}{\operatorname{\text{\bf I}}}
\newcommand{\n}{m}
\newcommand{\Erf}{\operatorname{erf}}
\newcommand{\Erfc}{\operatorname{erfc}}
\newcommand{\hatR}{V}
\newcommand{\hatRDef}{T_{s}}
\newcommand{\U}{\Upsilon}
\newcommand{\Pot}{\mathcal{U}_0}
\newcommand{\PotV}{\mathscr{U}_0}
\newcommand{\Potu}{\mathcal{U}}
\newcommand{\PotDef}{\Xi^{\circledast}}
\newcommand{\dPot}{Q^{\diamond}}
 
\newcommand{\E}{E}
\newcommand{\calH}{\mathscr H}
\newcommand{\C}{\mathbb{C}}
\newcommand{\D}{\mathbb{D}}
\newcommand{\T}{\mathbb{T}}
\newcommand{\Te}{\mathbb{T}}
\newcommand{\N}{\mathbb{N}}
\newcommand{\Z}{\mathbb{Z}}
\newcommand{\R}{\mathbb{R}}
\newcommand{\A}{\mathbb{A}}
\newcommand{\kerK}{\mathcal{K}}
\newcommand{\Pol}{\mathrm{Pol}}
\newcommand{\calA}{\mathcal{A}}
\newcommand{\calB}{\mathcal{B}}
\newcommand{\calS}{\mathscr{S}}
\newcommand{\calV}{\mathscr{V}}
\newcommand{\calD}{\mathscr{D}}
\newcommand{\calK}{\mathcal{K}}
\newcommand{\calE}{\mathcal{E}}

\newcommand{\Fspace}{\mathscr{F}(\C)}
\newcommand{\Fshort}{\mathscr{F}}
\newcommand{\diffA}{\mathrm{dA}}
\newcommand{\diffs}{\mathrm{ds}}
\newcommand{\diff}{\mathrm{d}}
\newcommand{\Pop}{\mathbf{P}}

\newcommand{\Sop}{\mathbf{S}}
\newcommand{\Mop}{\mathbf{M}}
\newcommand{\Hop}{\mathbf{H}}
\newcommand{\Iop}{\mathbf{I}}
\newcommand{\Rop}{\mathbf{R}}
\newcommand{\Uop}{\mathbf{U}}
\newcommand{\Bop}{\mathbf{B}}
\newcommand{\Jop}{\mathbf{J}}

\newcommand{\say}[1]{``#1''}
\newcommand{\e}{\mathrm{e}}
\newcommand{\Ordo}{\mathrm{O}}
\newcommand{\ordo}{\mathrm{o}}
\newcommand{\imag}{\mathrm{i}}
\renewcommand{\Re}{\mathrm{Re}\,}

\renewcommand{\hat}{\widehat}
\renewcommand{\k}{\{k\}}

\newcommand{\re}{\mathrm{Re}}
\newcommand{\im}{\mathrm{Im}}

\newcommand{\kn}{\langle k_{\n}\rangle}

\newcommand\scalemath[2]{\scalebox{#1}{\mbox{\ensuremath{\displaystyle #2}}}}

\newcommand{\calU}{\mathcal{U}}
\newcommand{\calY}{\mathcal{Y}}
\newcommand{\Lop}{\mathbf{L}}

\DeclareFontFamily{U}{rcjhbltx}{}
\DeclareFontShape{U}{rcjhbltx}{m}{n}{<->rcjhbltx}{}
\DeclareSymbolFont{hebrewletters}{U}{rcjhbltx}{m}{n}
\DeclareMathSymbol{\shin}{\mathord}{hebrewletters}{152}

\input pdf-trans
\newbox\qbox
\def\usecolor#1{\csname\string\color@#1\endcsname\space}
\newcommand\bordercolor[1]{\colsplit{1}{#1}}
\newcommand\fillcolor[1]{\colsplit{0}{#1}}
\newcommand\outline[1]{\leavevmode%
  \def\maltext{#1}%
  \setbox\qbox=\hbox{\maltext}%
  \boxgs{Q q 2 Tr \thickness\space w \fillcol\space \bordercol\space}{}%
  \copy\qbox%
}
\makeatother
\newcommand\colsplit[2]{\colorlet{tmpcolor}{#2}\edef\tmp{\usecolor{tmpcolor}}%
  \def\tmpB{}\expandafter\colsplithelp\tmp\relax%
  \ifnum0=#1\relax\edef\fillcol{\tmpB}\else\edef\bordercol{\tmpC}\fi}
\def\colsplithelp#1#2 #3\relax{%
  \edef\tmpB{\tmpB#1#2 }%
  \ifnum `#1>`9\relax\def\tmpC{#3}\else\colsplithelp#3\relax\fi
}
\bordercolor{black}
\fillcolor{white}
\def\thickness{.3}

\newtheorem{thm}{Theorem}[subsection]
\newtheorem{cor}[thm]{Corollary}
\newtheorem{lem}[thm]{Lemma}

\newtheorem{prop}[thm]{Proposition}

\theoremstyle{definition}
\newtheorem{defn}[thm]{Definition}

\newtheorem{prob}[thm]{Problem}

\theoremstyle{remark}
\newtheorem{rem}[thm]{Remark}

\usepackage{graphicx}

\numberwithin{equation}{subsection}
\numberwithin{figure}{subsection}

\title[Liouville phenomenon for the Klein-Gordon equation
\\
in $1+1$ dimensions]
{Liouville phenomenon for the Klein-Gordon equation
\\  in $1+1$ dimensions}

\begin{document}

\author[Haakan Hedenmalm]{Haakan Hedenmalm}

\address{Department  of Mathematics and Computer Sciences
\\
St. Petersburg State University
\\
St. Petersburg
\\
RUSSIA
\\
Beijing Institute for Mathematical Sciences and Applications
\\
Huairou District, Beijing 101408
\\
CHINA
}
\email{haakan00@gmail.com}



\date{\today}

\begin{abstract}
We study the Klein-Gordon equation in one spatial and one temporal
 dimension, focusing on the Liouville phenomenon in spacelike quarter-planes.
We consider solutions whose growth is restricted by a family of conditions
indexed by a parameter $q>0$, inspired by Young's inequality. Under suitable
growth restrictions, the values of a solution along the two characteristic
boundary rays are uniquely related. We establish sharp uniqueness criteria
in the different growth regimes, with $q=\frac12$ and $q=1$ marking two
transitions in the nature of the problem. The sharpness of the criteria,
including the critical equality cases, is demonstrated by constructing
nontrivial solutions. The resulting Liouville phenomenon has connections
with classical uniqueness principles for holomorphic and harmonic functions,
particularly the Phragm\'en--Lindel\"of principle.
\end{abstract}

\maketitle

\section*{Overview of the paper}

The introduction formulates the Liouville-type uniqueness problem for the
Klein-Gordon equation and presents the principal uniqueness and sharpness
theorems in Subsections 1.7 and 1.8, respectively. It also develops the
viewpoint of evolution along a characteristic direction, introduces stable
global solutions, and discusses the nonlinear connection with the
sinh-Gordon equation.

Section 2 collects the classical properties of the Klein-Gordon equation
needed in the sequel, including Riemann's solution of the Darboux--Goursat
problem and its decomposition into unilateral wave solutions. Section 3
treats the exponential growth case $q=1$. An auxiliary integral estimate
is established in Section 4, while Section 5 develops analytic uniqueness
methods for $0<q<\frac12$.

Section 6 establishes the results for $\frac12<q<1$, including the
sharpness construction. Section 7 treats the exponentially superlinear
regime $q>1$, using the generalized Laplace transform and moving
truncation. Section 8 addresses the critical exponent $q=\frac12$,
including the delicate nonuniqueness construction at equality. The
appendix supplies additional estimates and technical arguments concerning
stable solutions and the Riemann representation.

\section{Introduction}

\subsection{The uniqueness problem and its motivation}
The classical Darboux--Goursat problem for a hyperbolic equation
prescribes data on two intersecting characteristic curves. This suggests
a natural question: what remains of uniqueness if data are prescribed on
only one characteristic? For the Klein-Gordon equation, vanishing on one characteristic
semi-axis does not by itself determine a solution. We ask whether restrictions on growth at infinity can
replace the missing characteristic data.

We impose growth restrictions inspired by Young's inequality, which
suggests a natural balance between growth in the two characteristic
directions. This leads us to consider a family of growth conditions indexed
by a parameter $q>0$.
An important precursor is the exponential-growth case $q=1$.
Andrew Bakan obtained an unpublished uniqueness result in this setting
and communicated it privately to Maryna Viazovska and Danylo Radchenko.
A version was discussed in a Stockholm seminar in 2016 and later
mentioned in Viazovska's plenary lecture at the 2018 International
Congress of Mathematicians in Rio de Janeiro.
The present work investigates the corresponding uniqueness question
across a wider range of growth scales and determines sharp thresholds
in the regimes considered below.

Our main setting is the spacelike quarter-plane
$\R_{\ge0}\times\R_{\le0}$, with zero data along
$\{0\}\times\R_{\le0}$. The resulting uniqueness question is of
Liouville type. Although elementary to formulate, it exhibits
different behavior in the regimes
\[
0<q<\tfrac12,\qquad q=\tfrac12,\qquad
\tfrac12<q<1,\qquad q=1,\qquad q>1.
\]
The appropriate balance depends on $q$, and the values
$q=\frac12$ and $q=1$ mark transitions between distinct uniqueness regimes.
The sharp constants and the critical equality cases are determined by the
analysis developed below. The uniqueness criteria are complemented by
constructions of nontrivial solutions, which establish their sharpness,
including at the critical thresholds.

The growth conditions and their sharpness are the subject of the
theorems below. In particular, sufficiently restrictive growth in
the positive $x$-direction can force uniqueness even when the bound
is imposed only along an unbounded collection of horizontal lines;
other regimes involve a quantitative balance between growth in
both coordinate directions. The optimality of the discretization
conditions is an interesting problem in its own right, which we leave
for future investigation.

Unlike characteristic initial--boundary value formulations, our
problem imposes no additional boundary condition to replace the
missing characteristic datum. Our methods connect hyperbolic
uniqueness with Laplace transforms, boundary uniqueness for
holomorphic functions, quasianalyticity, and conformal mapping.
The mechanism behind these uniqueness results is essentially geometric.
The growth restrictions on the solution induce estimates for an associated
holomorphic function, which identify regions where that function is bounded
or has at most polynomial growth. The decisive issue is whether the geometry
of these regions, together with the required decay along the real axis,
forces the holomorphic function to vanish. Although the relevant growth
scales and regions change with $q$, this comparison provides a common
principle behind the uniqueness arguments. The construction of nontrivial
solutions establishing sharpness requires additional arguments.

The different growth regimes lead to distinct complex-analytic problems:
in the subexponential regime, the relevant behavior of the Laplace
transform is near the origin, whereas the exponentially superlinear
regime calls for a generalized Laplace transform and an analysis near
infinity. Thus the changes in the growth conditions are reflected in
the analytic methods used to establish uniqueness.

We recall Riemann's solution of the classical Darboux--Goursat problem,
together with its uniqueness and invariance properties, in
Section \ref{sec:classical-KG} below.
Here we focus instead on the Laplace-transform evolution that motivates
our growth-dependent uniqueness criteria.
In characteristic coordinates the Klein-Gordon equation takes the form
\begin{equation}
u''_{xy}+u=0,\label{eq:KG-1}
\end{equation}
and we study continuous solutions in the spacelike quarter-plane.
For the Laplace-transform calculation below, we write
$u_{\mathrm{horz}}=\Rop[f,0]$ for the horizontal unilateral solution
with characteristic data $f$ on the horizontal axis and zero data
on the vertical axis. Its classical representation is established
in Section~\ref{sec:classical-KG}.

\subsection{The Phragm\'en--Lindel\"of principle and the hyperbolic contrast}
The analogous question for harmonic functions is expressed by the Phragm\'en--Lindel\"of principle. In the upper half-plane, vanishing boundary values can imply uniqueness when supplemented by sufficiently restrictive growth at infinity. We recall a convenient formulation.

\begin{thm}
{\rm (Phragm\'en-Lindel\"of for harmonic functions)}
Suppose $u:\mathfrak{H}\to\C$ is harmonic, where $\mathfrak{H}$ is the
half-plane $\mathfrak{H}=\R\times\R_{>0}$.
Suppose, moreover, that $u$ extends continuously to
the closed half-plane $\overline{\mathfrak{H}}=\R\times\R_{\ge0}$, and that
$u(x,0)=0$ holds for all $x\in\R$. Suppose, moreover, that there exists
a sequence of bounded domains $D_n$ in $\mathfrak{H}$, for $n=1,2,3,\ldots$,
which increase monotonically, $D_n\subset D_{n+1}$, and exhaust $\mathfrak{H}$
(i.e.,$\cup_n D_n=\mathfrak{H}$ holds) while
\[
u(x,y)=\ordo(y),\qquad (x,y)\in\bigcup_{n=1}\partial D_n,
\]
holds as $|(x,y)|\to+\infty$. Then $u(x,y)\equiv0$ identically.
\label{thm:PhL}
\end{thm}

The proof rests on the maximum principle, applied separately to the real and imaginary parts. The harmonic function $u(x,y)=y$ shows the relevance of the growth restriction. By Schwarz reflection, the boundary condition also connects this principle with the classical Liouville theorem. 

The distinction between elliptic and hyperbolic equations is fundamental.
For elliptic equations, the maximum principle provides a powerful connection
between boundary behavior and uniqueness, underlying the classical
Phragm\'en--Lindel\"of principle. Hyperbolic equations have a different
character. Their solutions propagate along characteristics, and no comparable
maximum principle is available. Nevertheless, suitable growth restrictions
may compensate for incomplete characteristic data and restore uniqueness.
It is in this sense that the Phragm\'en--Lindel\"of principle serves as a
model for the results developed here.

\subsection{The Klein-Gordon equation and characteristic geometry}
The characteristic axes divide the plane into timelike regions $xy>0$
and spacelike regions $xy<0$. Figure~\ref{fig:spacelike-quarter-plane}
shows the two spacelike quarter-planes. The characteristic lines $xy=0$
form the light cone through the origin. Our principal uniqueness problem
concerns the spacelike quarter-plane $\R_{\ge0}\times\R_{\le0}$,
with prescribed zero data on its vertical characteristic boundary.

The distinction between characteristic data and uniqueness under growth
restrictions is also reflected in earlier work. Ullrich \cite{Ullrich}
studied uniqueness of extension from a characteristic hypersurface for
the Klein-Gordon equation in a higher-dimensional setting, under
growth assumptions appropriate to tempered distributions. Heinzl and
Werner \cite{Heinzl-Werner} considered the Klein-Gordon equation as
an initial-boundary value problem in light-front formulations of
relativistic field theory, emphasizing the role of additional boundary
conditions. Our question is different: we investigate when growth
restrictions at infinity can restore uniqueness without prescribing a
second characteristic datum.

\begin{figure}[htbp]
\centering
\begin{tikzpicture}[scale=1.65]
  \fill[gray!18] (-2.0,0) rectangle (0,1.55);
  \fill[gray!18] (0,0) rectangle (2.0,-1.55);
  \draw[-latex,thick] (-2.2,0) -- (2.2,0) node[below] {$x$};
  \draw[-latex,thick] (0,-1.8) -- (0,1.8) node[left] {$y$};
  \node[above right] at (0,0) {$(0,0)$};
\end{tikzpicture}
\caption{The two spacelike quarter-planes $xy<0$, shaded in gray. The coordinate axes form the characteristic boundary separating the spacelike and timelike regions.}
\label{fig:spacelike-quarter-plane}
\end{figure}

\subsection{The evolution equation for horizontal unilateral wave solutions}
\label{subsec:laplace-evolution}
We consider a general horizontal unilateral wave solution
$u_{\mathrm{horz}}=\Rop[f,0]$, where $f$ is a continuous function on the
line with $f(0)=0$. These solutions occur naturally in the context of
Riemann's solution to the Darboux-Goursat problem, see Subsection
\ref{ss:splitting} below for the details. Explicitly, $u_{\mathrm{horz}}$ is
given by
\begin{equation}
u_{\mathrm{horz}}(x,y)=f(x)-\int_0^x J_{0,1}(x-s,y)f(s)\,\diff s.
\label{eq:Ropf0-1}
\end{equation}
Here, we use the notation $J_{0,1}$ (and alternatively $J_{1,0}$)
for the Bessel-type entire function of two variables given by the
power series
\[
J_{0,1}(x,y)=J_{1,0}(y,x):=\sum_{j=0}^{+\infty}\frac{(-1)^j}{j!(j+1)!}\,x^{j}y^{j+1}.
\]
The formula \eqref{eq:Ropf0-1} is of convolution type on the half-line, which
suggests that we might try to analyze it using the classical Laplace transform.
To this end, we assume that $f$ does not grow faster than exponentially in
the mean. More explicitly, we require that 
\[
\forall\sigma>\sigma_0:\quad
\int_0^{+\infty}\e^{-x\sigma}|f(x)|\diff x<+\infty,
\]
holds for some parameter value $\sigma_0\in\R_{\ge0}$.
Then the Laplace transform $\Lop[f]$ given by
\[
\Lop[f](\zeta)=\int_0^{+\infty}\e^{-x\zeta}f(x)\,\diff x,\qquad \Re \zeta>\sigma_0,
\]
defines a holomorphic function in the given half-plane, and if we write
\[
[J_{0,1}(\cdot,y)\ast f](x)=
\int_0^x J_{0,1}(x-s,y)f(s)\,\diff s
\]
for the semigroup convolution, the Laplace transform of the convolution
is just the product of the two individual Laplace transforms,
\[
\Lop[J_{0,1}(\cdot,y)\ast f](\zeta)=\Lop[J_{0,1}(\cdot,y)](\zeta)\Lop[f](\zeta),
\]
where
\[
\Lop[J_{0,1}(\cdot,y)](\zeta)=\int_0^{+\infty}\e^{-x\zeta}J_{0,1}(x,y)\,\diff x
=1-\e^{-y/\zeta},\qquad \re\zeta>0.
\]
Although known, this calculation is surprisingly hard to find in the literature.
It is, however, rather straightforward:
\begin{multline}
\int_0^{+\infty}\e^{-x\zeta}J_{0,1}(x,y)\,\diff x
=\sum_{j=0}^{+\infty}\frac{(-1)^j}{j!(j+1)!}
\int_{0}^{+\infty}x^j\,y^{j+1}\e^{-x\zeta}\diff x
\\
=\sum_{j=0}^{+\infty}\frac{(-1)^j}{(j+1)!}y^{j+1}\zeta^{-j-1}=1-\e^{-y/\zeta},
\end{multline}
again provided that $\re\zeta>0$.
It is now immediate from \eqref{eq:Ropf0-1} that
\begin{multline}
\Lop[u_{\mathrm{horz}}(\cdot,y)](\zeta)=
\int_0^{+\infty}\e^{-x\zeta} u_{\mathrm{horz}}(x,y)\,\diff x
=\Lop[f](\zeta)-\Lop[J_{0,1}(\cdot,y)\ast f](\zeta)
\\
= \Lop[f](\zeta)-\Lop[J_{0,1}(\cdot,y)](\zeta)\Lop[f](\zeta)
=\e^{-y/\zeta}\Lop[f](\zeta),\qquad \re\zeta>\sigma_0.
\label{eq:Ropf0-2}
\end{multline}
The identity \eqref{eq:Ropf0-2} is fundamental to our analysis of the
horizontal unilateral wave solutions $u_{\mathrm{horz}}=\Rop[f,0]$.
The guiding viewpoint is to regard the Klein-Gordon equation as defining
an evolution family along a characteristic direction, with data on the
transverse characteristic semi-axis serving as initial data.
From this perspective, our Liouville-type theorems describe how
anisotropic growth restrictions constrain this evolution, leading to
distinct analytic mechanisms and sharp uniqueness thresholds.
The simple multiplicative evolution in the $y$-parameter is reminiscent
of the evolution of spectral data in the inverse scattering method for
integrable equations such as KdV; see Beals, Deift, and Tomei
\cite{BDT} and Deift and Trubowitz \cite{DeiftTrubowitz}.
For an introduction to related ideas, see also \cite{BPT}.
The unified transform method for linear boundary-value problems offers
another spectral perspective, as discussed by Mantzavinos \cite{Mantzavinos2015}.
This is an analogy between evolution laws, not an identification of the
present Laplace transform with nonlinear scattering data.
For growth bounds with $0<q\le1$, the characteristic datum has at most
exponential growth in the positive $x$-direction, so the ordinary Laplace
transform provides the natural starting point. For $q>1$, this need no
longer be the case, and a generalized transform is required later.

\subsection{Liouville-type theorems for spacelike quarter-planes}
We proceed to study what kind of growth condition on a solution $u=\Rop[f,0]$
will be able to force $u$ to collapse to $0$ (so that $f$ collapses too)
in the quarter-plane $\R_{\ge0}\times\R_{\le0}$. A useful guide is the Bessel growth scale exemplified by the globally
 defined solution
\[
u_{(0,0)}(x,y)=J_{0,0}(x,y)=J_0(2\sqrt{xy})
=\sum_{n=0}^{+\infty}\frac{(-xy)^n}{(n!)^2},
\qquad (x,y)\in\R^2.
\]
Here the Bessel expression is understood through its entire power series.
In the spacelike quarter-plane $xy<0$, the characteristic exponential
growth scale is $\exp(2\sqrt{|xy|})$; see
Section~\ref{sec:classical-KG} for the asymptotics. More generally, before prescribing a growth rate in the second
characteristic direction, we may consider bounds of the form
\begin{equation}
|u(x,y)|\le C\exp\bigl(\theta_1 x^q+\Lambda(y)\bigr),
\qquad (x,y)\in\R_{\ge0}\times\R_{\le0},
\label{eq:general-Lambda-growth}
\end{equation}
where $q>0$, $\theta_1>0$, and $\Lambda$ is a real-valued growth
majorant on $\R_{\le0}$ to be determined. The Lorentz transformations
$(x,y)\mapsto(\lambda x,\lambda^{-1}y)$, $\lambda>0$, preserve the
Klein-Gordon equation. They transform the growth parameter and majorant according to
\[
\theta_1\mapsto\lambda^q\theta_1,
\qquad \Lambda(y)\mapsto\Lambda(\lambda^{-1}y).
\]
Thus any intrinsic uniqueness criterion should respect this scaling.
For a power-law majorant $\Lambda(y)=\theta_2|y|^r$, the combination
$\theta_1^{1/q}\theta_2^{1/r}$ is invariant. But Lorentz invariance does not itself specify the proposed exponent $r$.
When $q>\frac12$, the Bessel growth scale and Young's inequality
would suggest $r=q''=q/(2q-1)$. For $0<q\le\frac12$,
this conjugate-exponent prescription is unavailable, but the general
growth question \eqref{eq:general-Lambda-growth} remains meaningful.

This motivates applying the classical \emph{Young's inequality for products},
which, in terms of $q$ and $q''=q/(2q-1)$, reads
\begin{equation}
ab\le \frac{a^{2q}}{2q}+\frac{b^{2q''}}{2q''},
\qquad 0\le a,b<+\infty,
\label{eq:Young-1}
\end{equation}
provided that $q>\frac12$. Here
$\frac{1}{2q}+\frac{1}{2q''}=1$.
Equality holds in \eqref{eq:Young-1} if and only if
$a^{2q}=b^{2q''}$.

We introduce two positive real parameters $\theta_1,\theta_2$, and put
\[
a=(\theta_1 q)^{1/(2q)}\sqrt{|x|},\quad b=(\theta_2 q'')^{1/(2q'')}\sqrt{|y|},
\]
and apply the inequality \eqref{eq:Young-1}:
\begin{equation}
(\theta_1 q)^{1/(2q)}\sqrt{|x|}(\theta_2 q'')^{1/(2q'')}\sqrt{|y|}
\le \theta_1\frac{|x|^q}{2}+\theta_2\frac{|y|^{q''}}{2},
\label{eq:Young-2}
\end{equation}
and then multiply by $2$ on both sides:
\begin{equation}
2\,(q\theta_1)^{1/(2q)}(q''\theta_2)^{1/(2q'')}\sqrt{|xy|}
\le \theta_1|x|^q +\theta_2|y|^{q''}.
\label{eq:Young-3}
\end{equation}
Suppose now that the parameters $\theta_1,\theta_2$ satisfy the strict
inequality
\begin{equation}
\theta_1^{1/q}\theta_2^{1/q''}<q^{-1/q}(q'')^{-1/q''},
\label{eq:Young-4}
\end{equation}
and consider the uniform growth condition
\begin{equation}
|u(x,y)|=|u_{\mathrm{horz}}(x,y)|=\Ordo\big(
\exp(\theta_1|x|^q+\theta_2|y|^{q''})\big),\qquad (x,y)\in\R_{\ge0}\times\R_{\le0},
\label{eq:Young-5}
\end{equation}
which corresponds to the natural choice
$\Lambda(y)=\theta_2|y|^{q''}$ of the growth majorant introduced above.
The sharpness of Young's inequality shows that, under \eqref{eq:Young-4},
the bound \eqref{eq:Young-5} cannot hold for the Bessel solution
$u_{(0,0)}$: along suitable rays in the spacelike quarter-plane, its
exponential growth exceeds the proposed majorant. This motivates the
following uniqueness problem.

\begin{prob}
Suppose $u=u_{\mathrm{horz}}$ as defined by \eqref{eq:Ropf0-1}, with continuous characteristic datum $f$ satisfying $f(0)=0$. Suppose, moreover,
that $u$ meets the growth condition \eqref{eq:Young-5} for some positive
real parameters $q,\theta_1,\theta_2$ with $q>\frac12$. Given $q$, what is
the relevant condition on $\theta_1,\theta_2$ which assures us that $u$
collapses to $u=0$ on $\R_{\ge0}\times\R_{\le0}$? Is it enough that \eqref{eq:Young-4}
holds, or do we need more? 
\end{prob}

The requirement that the growth bound hold throughout the spacelike
quarter-plane may be unnecessarily restrictive. It is natural to ask
whether uniqueness can already be obtained when the bound is imposed
only on $\R_{\ge0}\times\mathcal{Y}$, where
$\mathcal{Y}\subset\R_{\le0}$ is a set unbounded below.

The answer depends on the growth regime and on how densely the horizontal
levels in $\mathcal{Y}$ cover the negative semi-axis. For $\frac12<q<1$,
Theorem \ref{thm:main-1-2} below gives a sharp sufficient condition when
$\mathcal{Y}$ induces an asymptotic $q$-covering. In that setting the
coefficient condition is
\[
\theta_1^{1/q}\theta_2^{1/q''}<
q^{-1/q}(q'')^{-1/q''}\sin^2\frac{\pi}{2q}.
\]
The case $q=1$ is treated separately in Theorem \ref{thm:main-1-0},
where an unbounded set of horizontal levels suffices and the coefficient
condition is $\theta_1\theta_2<1$. These different assumptions on
$\mathcal{Y}$ are important and should not be conflated.

\subsection{The various growth regimes}

The uniqueness problem takes different forms depending on the growth
exponent $q$. For $0<q<\frac12$, sufficiently slow growth in the positive
$x$-direction can force uniqueness without a separate prescribed growth
bound in the negative $y$-direction. This phenomenon is connected with
analytic non-quasianalyticity and the work of Beurling and Korenblum.
At the critical exponent $q=\frac12$, a different borderline growth
condition is needed.

For $\frac12<q<1$, Young's inequality suggests the conjugate growth
exponent $q''=q/(2q-1)$ in the second characteristic direction. In this
range, the distribution of the horizontal levels on which growth is
controlled also plays a role. At $q=1$, the growth is exponential in
both characteristic directions, with a product condition on the
coefficients. Finally, for $q>1$, growth in the positive $x$-direction
is exponentially superlinear, and a generalized Laplace transform is
needed.
Young's inequality applies throughout the
entire range $\frac12<q<+\infty$. The transition at $q=1$ does not
reflect a limitation of Young's inequality. Rather, the analytic
methods change: the range $\frac12<q<1$ is approached through the
ordinary Laplace transform and boundary uniqueness, the exponential
case $q=1$ requires a separate argument, and the regime $q>1$ calls
for the generalized Laplace transform.

The principal features of the five growth regimes are summarized in
the following table. The table describes the analytic framework rather
than replacing the precise hypotheses in the theorems below.
\medskip
\begin{center}
\small
\begin{tabular}{@{}lll@{}}
\hline
Range of $q$ & Growth framework & Main analytic method \\
\hline
$0<q<\frac12$ & Control in $x$ alone & Hayman--Korenblum \\
$q=\frac12$ & Critical growth scale & Ahlfors--Warschawski \\
$\frac12<q<1$ & $\Lambda(y)=\theta_2|y|^{q''}$ & Boundary uniqueness \\
$q=1$ & Exponential growth & Interior uniqueness \\
$q>1$ & Exponentially superlinear & Generalized Laplace transform \\
\hline
\end{tabular}
\end{center}
\medskip

\subsection{Liouville-type theorems in the various growth regimes}

We now state our principal Liouville-type theorems, treating separately the five growth regimes discussed in the preceding subsection.

In the case $\mathcal{Y}=\R_{\le0}$, the following theorem is due to Andrew Bakan (unpublished).

\begin{thm}
Suppose $u$ is a continuous complex-valued function on the closed
spacelike quarter-plane $\R_{\ge0}\times\R_{\le0}$, which solves the Klein-Gordon
equation $u''_{xy}+u=0$ in the sense of distribution theory in the interior
quarter-plane $\R_{>0}\times\R_{<0}$. 
Suppose moreover that $u$ vanishes along the semi-axis $\{0\}\times \R_{\le0}$.
Finally, we assume that $u$ meets the uniform growth bound
\[
|u(x,y)|=\Ordo\big(\exp(\theta_1 x+\theta_2|y|)\big),\qquad
(x,y)\in\R_{\ge0}\times\mathcal{Y},  
\]
where $\theta_1,\theta_2\in\R_{>0}$ and the nonempty subset
$\mathcal{Y}\subset\R_{\le0}$ satisfies
$\inf\mathcal{Y}=-\infty$.
Then, if $\theta_1\theta_2<1$, 
it follows that the solution collapses to $u(x,y)=0$ for all
$(x,y)\in\R_{\ge0}\times \R_{\le0}$. 
\label{thm:main-1-0}
\end{thm}

The proof of Theorem \ref{thm:main-1-0} is supplied in Subsection
\ref{ss:q=1}. A reasonable interpretation of the theorem is that for solutions 
relevant to Physics, the values along the negative $y$-axis determine the 
values along the positive $x$-axis, and vice versa. 

\begin{rem}
(a) If we compare with the Phragm\'en--Lindel\"of principle of Theorem \ref{thm:PhL}, the union of lines
$\R_{\ge0}\times\mathcal{Y}$ corresponds to the union of fronts
$\cup_n\partial D_n$.

\noindent{(b)} The fact that the criterion involves the product $\theta_1\theta_2$
is a consequence of the Lorentz group invariance of the Klein-Gordon equation
(see Proposition \ref{prop:Poincare-1}). 

\end{rem}

The condition that $u$ vanishes along the semi-axis $\{0\}\times\R_{\le0}$
entails that $u$ may be extended to a solution to the Klein-Gordon equation
in the whole plane $\R^2$ which vanishes on the line $\{0\}\times\R$, that is,
the extension is a horizontal unilateral solution. 

It is a natural question to ask if we can play with the growth conditions 
while keeping the conclusion of Theorem \ref{thm:main-1-0}. For instance, if
the growth in the $x$ direction is slower, can we allow for faster growth
in the $y$ direction? It turns out that this question has answers which
connect with the foundations of harmonic analysis, and more precisely,
with the issue of so-called analytic non-quasianalyticity, as studied by
Arne Beurling \cite{Beu}.

\begin{thm}
Suppose $u$ is a continuous complex-valued function on the closed
spacelike quarter-plane $\R_{\ge0}\times\R_{\le0}$, which solves the Klein-Gordon
equation $u''_{xy}+u=0$ in the sense of distribution theory in the interior
quarter-plane $\R_{>0}\times\R_{<0}$.
Suppose moreover that $u$ vanishes along the semi-axis $\{0\}\times \R_{\le0}$.
Finally, we assume that the nonempty subset $\mathcal{Y}\subset\R_{\le0}$ has
$\inf\mathcal{Y}=-\infty$, and 
that $u$ enjoys the growth bound
\[
\forall y\in\mathcal{Y}:\,\,\,
|u(x,y)|=\Ordo_y\big(\exp(x^{q})\big),
\]
as $x\to+\infty$, where the implied constant is allowed to depend on $y$, and
the exponent $q$ satisfies $0<q<\frac12$.
It then follows that the solution collapses to $u=0$ on
$\R_{\ge0}\times\R_{\le0}$. 
\label{thm:main-2}
\end{thm}

The normalization of the exponential coefficient to $1$ entails no loss of generality: for any $\theta>0$ and $0<q<\frac12$, one may choose $q<\tilde{q}<\frac12$, so that $\theta x^q\le x^{\tilde{q}}$ for all sufficiently large $x$.

The proof of Theorem \ref{thm:main-2} is based on the work of Hayman and
Korenblum \cite{HayKor0}, \cite{HayKor}, and on Korenblum's address at the International Congress of Mathematicians in Helsinki in 1978 \cite{Kor3},
which in turn mentions the connection with Beurling's contribution. It is supplied in 
Subsection \ref{ss:Korenblum-1}.

In Theorem \ref{thm:main-2}, we do not require any particular growth control on $u(x,y)$
in the negative $y$ direction, and the price we pay is a rather strict growth
bound in the positive $x$ direction. In Theorem \ref{thm:main-1-0}, on the
other hand, we have exponential control in both directions. More generally,
we may distribute the growth restrictions between the two coordinate
directions according to the conjugate exponents $q$ and $q''$ suggested
by Young's inequality. To formulate the corresponding uniqueness theorem,
we need a condition describing how the set $\mathcal{Y}$ is distributed
along the negative $y$-axis.

\begin{defn}
For $\frac12<q<1$ and $q'':=q/(2q-1)$, a subset $\calY\subset\R_{\le0}$ is said 
to \emph{induce an asymptotic $q$-covering} of $\R_{\le0}$ 
if there exists a constant $M\in\R_{>0}$ 
such that
\[
]-\infty,y_1]\subset\bigcup_{y\in\calY}[y-MR(y,q),y+MR(y,q)]
\] 
holds for some $y_1\in\R_{<0}$, where
\[
R(y,q):=
|y|\,\min\{1,|y|^{-\frac{q''}{2}}\}.
\]
\end{defn}

\begin{rem}
For $\frac12<q<1$ we know that $1<q''<+\infty$. For $q$ close to $1$, 
$q''$ gets close to $1$, so that this condition allows 
for rather sparse sets $\calY$, whereas for $q$ close to $\frac12$, $q''$
gets close to $+\infty$, and then the set $\calY$ has to be much denser. 
\end{rem}

\begin{thm}
Suppose $u$ is a continuous complex-valued function in the quarter-plane
$\R_{\ge0}\times\R_{\le0}$, which solves the Klein-Gordon equation
$u''_{xy}+u=0$ in the sense of distribution theory in the interior quarter-plane
$\R_{>0}\times\R_{<0}$.
Suppose moreover that $u$ vanishes along the semi-axis $\{0\}\times \R_{\le0}$.
Finally, we assume that $u$ meets the uniform growth bound
\[
|u(x,y)|=\Ordo\big(\exp(\theta_1|x|^{q}+\theta_2|y|^{q''})\big),\qquad
(x,y)\in\R_{\ge0}\times\calY,  
\]
where $0<\theta_1,\theta_2<+\infty$, $\frac12<q<1$ and $q''=q/(2q-1)$,
where the subset $\calY\subset\R_{\le0}$ is supposed to induce 
an asymptotic $q$-covering of $\R_{\le0}$. If we have the bound
\[
\theta_1^{1/q}\theta_2^{1/q''}<\bigg(\frac{1}{q}\bigg)^{1/q}
\bigg(\frac{1}{q''}\bigg)^{1/q''}
\sin^2\frac{\pi}{2q},
\]
it then follows that the solution collapses to $u=0$ on all of
$\R_{\ge0}\times\R_{\le0}$. 
\label{thm:main-1-2}
\end{thm}

The proof of this theorem is supplied in Section \ref{sec:skewed-1}.

\begin{rem}\label{rem:observation-translation}
In the proofs of Theorems \ref{thm:main-1-0} and \ref{thm:main-1-2},
we may, for simplicity, assume that $0\in\calY$. Indeed, fix
$y_0\in\calY$ and replace $u(x,y)$ by $v(x,y):=u(x,y_0+y)$,
and $\calY$ by $\tilde{\calY}:=(\calY-y_0)\cap\R_{\le0}$.
Then $0\in\tilde{\calY}$, and the translated set is unbounded
below. In the setting of Theorem \ref{thm:main-1-2}, it also
induces an asymptotic $q$-covering: for fixed $y_0$, the ratio
$R(y-y_0,q)/R(y,q)$ tends to $1$ as $y\to-\infty$, so that
translating the covering intervals only requires an enlargement
of the covering constant. The growth bound for $v$ has the same
parameters in Theorem \ref{thm:main-1-0}. In Theorem
\ref{thm:main-1-2}, it holds with $\theta_2$ replaced by
$(1+\epsilon)\theta_2$ for any $\epsilon>0$, since
$|y+y_0|^{q''}\le(1+\epsilon)|y|^{q''}+C_{\epsilon,y_0}$.
The strict parameter inequality allows $\epsilon$ to be chosen
sufficiently small. Applying the theorem to $v$ gives $u=0$
for $y\le y_0$, and uniqueness for the characteristic initial-value
problem extends this conclusion to the whole quarter-plane.
\end{rem}

\begin{rem}
The strict parameter inequality in Theorem \ref{thm:main-1-2} is sharp
when the growth bound is imposed on the entire quarter-plane.
Indeed, the counterexample constructed in Subsection
\ref{sec:counter-skew} shows that, whenever
\[
\theta_1^{1/q}\theta_2^{1/q''}\ge
\bigg(\frac1q\bigg)^{1/q}
\bigg(\frac1{q''}\bigg)^{1/q''}
\sin^2\frac{\pi}{2q},
\]
there is a nontrivial solution vanishing on the vertical semi-axis
and satisfying the stated growth bound even for all $y\le0$.
In particular, the counterexample construction applies at equality as well.
Thus the strict inequality in Theorem \ref{thm:main-1-2} is the exact
threshold for uniqueness under the full quarter-plane growth bound.
\end{rem}

In terms of the parameter $q$ appearing in Theorems \ref{thm:main-2} and
\ref{thm:main-1-2}, the range $0<q<\frac12$ is covered by Theorem
\ref{thm:main-2}, while the range $\frac12<q<1$ constitutes Theorem
\ref{thm:main-1-2}, and $q=1$ amounts to Theorem \ref{thm:main-1-0}. 
The value $q=\frac12$ is not covered by any of these results, so it needs
to be treated separately.

\begin{thm}
Suppose $u$ is a continuous complex-valued function in the quarter-plane
$\R_{\ge0}\times\R_{\le0}$, which solves the Klein-Gordon equation
$u''_{xy}+u=0$ in the sense of distribution theory in the interior quarter-plane
$\R_{>0}\times\R_{<0}$.
Suppose moreover that $u$ vanishes along the semi-axis $\{0\}\times \R_{\le0}$.
Finally, we assume that $u$ meets the uniform growth bound
\[
|u(x,y)|=\Ordo\big(\exp\big(\theta_1\sqrt{x}+\exp(\theta_2\sqrt{|y|})\big)
\big),
\qquad
(x,y)\in\R_{\ge0}\times\R_{\le0},
\]
where the parameters $\theta_1,\theta_2\in\R_{>0}$ satisfy
$\theta_1\theta_2<2\pi$. 
It then follows that the solution collapses to $u=0$ on all of
$\R_{\ge0}\times\R_{\le0}$. 
\label{thm:main-3}
\end{thm}

The proof of this theorem is supplied in Section \ref{sec:critical-1}.

Endpoint cases in which a conjugate exponent becomes infinite often exhibit
exponential phenomena. Familiar examples include the John--Nirenberg
inequality for BMO and the Moser--Trudinger inequality for the Dirichlet
space, where the latter gives quadratic exponential integrability. In the
present setting, the endpoint $q=\frac12$ likewise exhibits a change in
the nature of the growth condition. Here the appropriate analogue is
pointwise growth of exponential type in $\sqrt{|y|}$.

\begin{rem}
(a) It is possible to replace the whole half-line $\R_{\le0}$ by a discrete subset $\calY$,
subject to a density condition, along the lines of Theorem \ref{thm:main-1-2}. This
time, the distance between consecutive points in $\calY$ should be on the order of
magnitude $\exp(-\frac12\theta_2|y|^{\frac12})$ for large negative $y$.

\noindent{(b)} The strict bound $\theta_1\theta_2<2\pi$ is sharp:
Subsection \ref{sec:counter-critical} constructs a nontrivial solution
satisfying the growth condition of Theorem \ref{thm:main-3} on the
entire quarter-plane whenever $\theta_1\theta_2\ge 2\pi$.
In particular, the borderline case $\theta_1\theta_2=2\pi$ also admits
nontrivial solutions; see Theorem \ref{prop:counter-critical-equality}.
\end{rem}

We also obtain a complete uniqueness--nonuniqueness dichotomy in the exponentially superlinear growth regime $q>1$, with an exact parameter threshold and nontrivial solutions even at equality.

\begin{thm}
\label{thm:superlinear-uniqueness}
Let $1<q<+\infty$ and put $q''=q/(2q-1)$. Suppose that $u$ is
continuous on $\R_{\ge0}\times\R_{\le0}$, solves $u''_{xy}+u=0$
distributionally in the interior, and vanishes on
$\{0\}\times\R_{\le0}$. Assume, as $x+|y|\to+\infty$ within the quarter-plane,
\[
 u(x,y)=\Ordo\!\left(\exp\big(\theta_1x^q+\theta_2|y|^{q''}\big)\right),
 \qquad \theta_1,\theta_2>0.
\]
If $\theta_1^{1/q}\theta_2^{1/q''}<\big(1/q\big)^{1/q}\big(1/q''\big)^{1/q''}$,
then $u=0$ identically.
\end{thm}
The proof is given in Section \ref{sec:superlinear-uniqueness}.
The critical factor satisfies
\[
\left(\frac1q\right)^{1/q}\left(\frac1{q''}\right)^{1/q''}\longrightarrow1
\qquad(q\to1^+),
\]
in agreement with Theorem \ref{thm:main-1-0}.
Together, Theorems \ref{thm:superlinear-uniqueness} and
\ref{thm:superlinear-sharpness} give the complete
uniqueness--nonuniqueness picture in the exponentially superlinear
regime $q>1$.

\begin{rem}
The transformation $(x,y)\mapsto(-y,-x)$ preserves the Klein-Gordon
equation and interchanges the two characteristic semi-axes. We may
therefore apply the preceding uniqueness results with the roles of the
coordinate directions reversed, obtaining corresponding statements in
which the prescribed data are taken on the other semi-axis.
For the Darboux--Goursat problem associated with the Klein-Gordon
 equation (1.1.1), it is instructive to compare the thresholds in
Theorems \ref{thm:main-1-2} and \ref{thm:superlinear-uniqueness}.
Writing $q''=q/(2q-1)$, their parameter inequalities have the same
form, except for the additional factor
$\sin^2(\pi/(2q))$ in Theorem \ref{thm:main-1-2}.
Both right-hand sides tend to $1$ as $q\to1$.
More importantly, the stronger criterion governing mutual uniqueness
ensures that either characteristic datum determines the other.
In the region between the two uniqueness thresholds, only one-sided
determination persists: one characteristic datum determines the other,
but the converse determination is no longer guaranteed.
This distinction concerns the direction of uniqueness for the
Darboux--Goursat problem, not merely the numerical constants.
\end{rem}

\begin{rem}
The Laplace-transform identity \eqref{eq:Ropf0-2}, discussed in
Subsection~\ref{subsec:laplace-evolution}, is part of the linear
evolution viewpoint. A nonlinear Volterra analogue for the sinh-Gordon
equation is described in Subsection~\ref{subsec:sg-laplace}.
In the superexponential growth regime, the ordinary Laplace transform
of the initial data need not converge in any half-plane. Section~\ref{sec:superlinear-uniqueness}
instead uses a generalized Laplace transform with a moving truncation.
\end{rem}

\subsection{Sharpness of the Liouville-type theorems}

We next address the sharpness of the preceding uniqueness results by
exhibiting nontrivial solutions when the relevant coefficient conditions
fail. We begin with the more elementary exponential and superlinear
cases, then turn to the intermediate regime and its critical endpoint.

The exponential growth criterion for $q=1$ is sharp, including at equality.

\begin{thm}
Suppose $\theta_1,\theta_2\in\R_{>0}$ with $\theta_1\theta_2\ge1$.
There exists a nontrivial continuous horizontal unilateral solution 
$u=u_{\mathrm{horz}}$ to the Klein-Gordon equation $u''_{xy}+u=0$ in the sense of
distribution theory, such that the restriction of $u$ to the quarter-plane
$\R_{\ge0}\times\R_{\le0}$ does not vanish identically, while it meets the
uniform growth condition
\[
|u(x,y)|=\Ordo\big(\exp(\theta_1 x+\theta_2|y|)\big),\qquad
(x,y)\in\R_{\ge0}\times\R_{\le0}.  
\]
\label{thm:main-1-1}
\end{thm}

The proof of Theorem \ref{thm:main-1-1} is supplied in Subsection 
\ref{ss:growth-1}.

The superlinear regime $q>1$ also admits nontrivial solutions at the threshold.

\begin{thm}
\label{thm:superlinear-sharpness}
Let $q>1$, put $q''=q/(2q-1)$, and let $\theta_1,\theta_2>0$ satisfy
\[
 \theta_1^{1/q}\theta_2^{1/q''}\ge
 \left(\frac1q\right)^{1/q}\left(\frac1{q''}\right)^{1/q''}.
\]
Then there exists a nontrivial continuous horizontal unilateral solution 
$u=u_{\mathrm{horz}}$ to the Klein-Gordon equation $u''_{xy}+u=0$ in the sense of
distribution theory, such that the restriction of $u$ to the quarter-plane
$\R_{\ge0}\times\R_{\le0}$ does not vanish identically, while it meets the
uniform growth condition
\[
|u(x,y)|=\Ordo\big(\exp(\theta_1x^q+\theta_2|y|^{q''})\big),
\qquad (x,y)\in\R_{\ge0}\times\R_{\le0}.
\]
In particular, nonuniqueness occurs also at equality in the parameter condition.
\end{thm}
The proof is supplied in Section \ref{sec:superlinear-uniqueness}.

For $\frac12<q<1$, the following theorem establishes sharpness of the intermediate-regime criterion, including equality.

\begin{thm}
\label{thm:intermediate-sharpness}
Let $\frac12<q<1$, $q''=q/(2q-1)$, and $\theta_1,\theta_2>0$ satisfy
\[
\theta_1^{1/q}\theta_2^{1/q''}\ge
\left(\frac1q\right)^{1/q}\left(\frac1{q''}\right)^{1/q''}
\sin^2\frac{\pi}{2q}.
\]
Then there exists a nontrivial continuous horizontal unilateral solution
$u=u_{\mathrm{horz}}$ to the Klein-Gordon equation $u''_{xy}+u=0$ in the
sense of distribution theory, such that the restriction $u$ to the
quarter-plane $\R_{\ge0}\times\R_{\le0}$ does not vanish identically while it
meets the uniform growth condition
\[
|u(x,y)|=\Ordo\bigl(\exp(\theta_1x^q+\theta_2|y|^{q''})\bigr),
\qquad(x\ge0,\ y\le0).
\]
\end{thm}

The proof is supplied in Subsection \ref{sec:counter-skew}.

At the critical exponent $q=\frac12$, we first give the nonuniqueness
statement under the strict reverse inequality.

\begin{thm}
\label{eq:counter-critical-prop:counterexample}
Suppose that $\theta_1,\theta_2\in\R_{>0}$ are given with
$\theta_1\theta_2>2\pi$. 
Then there exists a nontrivial continuous horizontal unilateral solution
$u=u_{\mathrm{horz}}$ to the Klein-Gordon equation $u''_{xy}+u=0$ in the
sense of distribution theory, such that the restriction $u$ to the
quarter-plane $\R_{\ge0}\times\R_{\le0}$ does not vanish identically while it
meets the uniform growth condition
\[
 |u(x,y)|=\Ordo\bigl(\exp(\theta_1\sqrt x+\exp(\theta_2\sqrt{|y|}))\bigr),
 \qquad (x,y)\in\R_{\ge0}\times\R_{\le0}.
\]
\end{thm}

The borderline case at $q=\frac12$ is covered by the next theorem.

\begin{thm}
\label{prop:counter-critical-equality}
Suppose that $\theta_1,\theta_2\in\R_{>0}$ are given with
$\theta_1\theta_2=2\pi$. 
Then there exists a nontrivial continuous horizontal unilateral solution
$u=u_{\mathrm{horz}}$ to the Klein-Gordon equation $u''_{xy}+u=0$ in the
sense of distribution theory, such that the restriction $u$ to the
quarter-plane $\R_{\ge0}\times\R_{\le0}$ does not vanish identically while it
meets the uniform growth condition
\[
|u(x,y)|=\Ordo\bigl(\exp(\theta_1\sqrt x+\exp(\theta_2\sqrt{|y|}))\bigr),
\qquad (x,y)\in\R_{\ge0}\times\R_{\le0}.
\]
Consequently, the conclusion of Theorem
\ref{eq:counter-critical-prop:counterexample} holds whenever
$\theta_1\theta_2\ge2\pi$.
\end{thm}

The proofs of these two critical-regime theorems are supplied in Subsection
\ref{sec:counter-critical}.

\subsection{The stable global solution for data along an axis}

The Liouville-type uniqueness phenomenon described in the preceding subsection
suggests that we should consider a class of global solutions for data along a
characteristic axis, which we call the \emph{stable solution}.

The stable solution is closely related to the $L^1$-mixed bounded exponential
solutions considered in \cite{BHM}. These have the spectral representation
\[
U[\varphi](x,y)=\int_{\R}\exp(\imag tx+\imag y/t)\varphi(t)\,\diff t,
\qquad \varphi\in L^1(\R).
\]
For such a solution, the initial datum along the $x$-axis is
$f(x)=\int_{\R}\exp(\imag tx)\varphi(t)\,\diff t$.
Thus, whenever the corresponding stable-solution integral converges and the
uniqueness statement below applies, $U[\varphi]$ agrees with
$\Sop_-[f]$. The stable-solution construction also makes sense for
classes of data not necessarily given by Fourier transforms of $L^1$
functions, so the two descriptions should not be identified without
additional hypotheses.

\begin{figure}[htbp]
\centering
\includegraphics[width=0.68\linewidth]{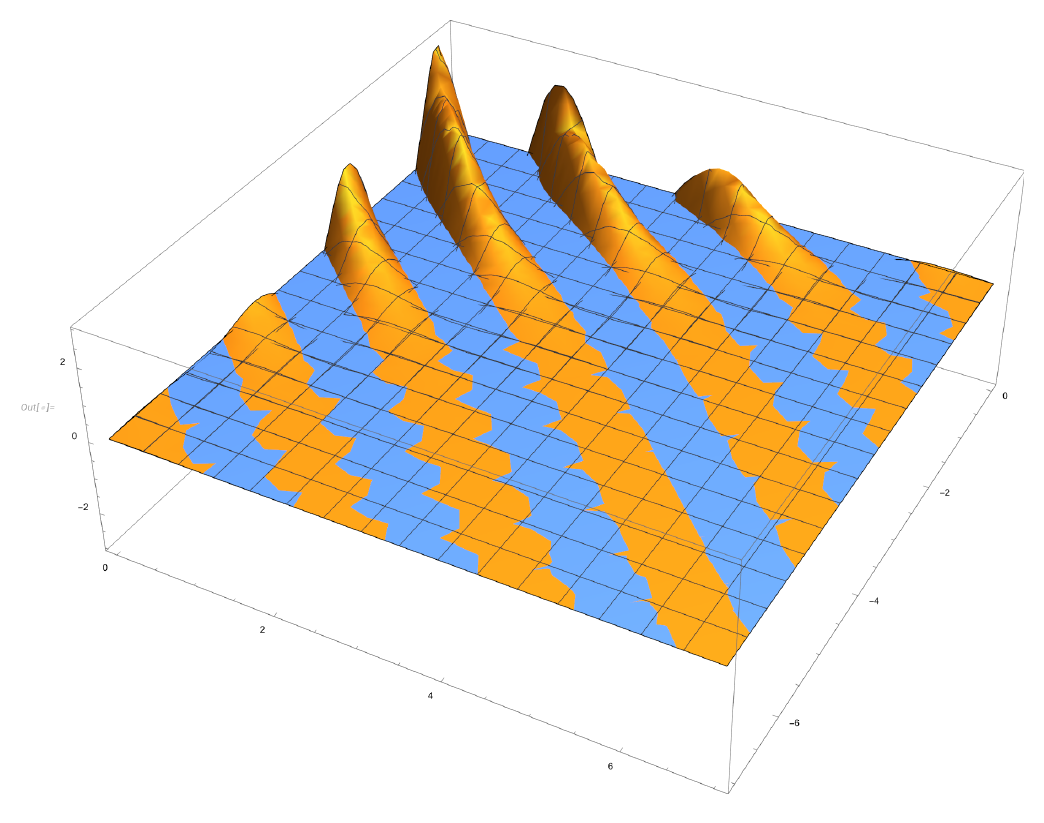}
\caption{The positive part of the globally bounded continuous solution $u^{(5,0)}(x,y)$ of the Klein-Gordon equation in the spacelike quarter-plane. The solution vanishes at every point of the discretized lattice-cross $(\pi\mathbb Z\times\{0\})\cup(\{0\}\times\pi\mathbb Z)$, except at $(5\pi,0)$, where it takes the value $1$. The notation $u^{(5,0)}$ and the illustration are taken from~\cite{BHM}.}
\label{fig:interpolation-solution}
\end{figure}

\begin{defn}
The \emph{stable solution for continuous data $f$ along the $x$-axis}
is the expression
\[
\Sop_-[f](x,y)=
\begin{cases}
f(x)-\int_{-\infty}^x J_{0,1}(x-t,y)f(t)\,\diff t\quad\text{if}\,\,\, y\ge0,
\\
f(x)+\int_{x}^{+\infty} J_{0,1}(x-t,y)f(t)\,\diff t\quad\text{if}\,\,\, y<0,
\end{cases}  
\]
provided the integrals converge.
\end{defn}

\begin{defn}
The \emph{stable solution for continuous data $g$ along the $y$-axis} is
the expression
\[
\Sop_\vert[g](x,y)=
\begin{cases}
g(y)-\int_{-\infty}^y J_{1,0}(x,y-t)g(t)\,\diff t\quad\text{if}\,\,\, x\ge0,
\\
g(y)+\int_{y}^{+\infty} J_{1,0}(x,y-t)g(t)\,\diff t\quad\text{if}\,\,\, x<0.
\end{cases}  
\]
provided the integrals converge.
\end{defn}

We note that in the regimes that appear in the stable solutions,
we may use the Bessel function asymptotics for $xy\gg 1$ to see that 
\[
J_{1,0}(x,y)=(x/y)^{\frac12}J_1(2\sqrt{xy})=\pi^{-1/2}x^{\frac14}y^{-\frac34}
\Big(\sin\Big(2\sqrt{xy}-\frac{\pi}{4}\Big)-\Ordo(xy)^{-\frac12}\Big). 
\]

\begin{rem}
Under the decay conditions on the data that for some $\epsilon>0$,
$f(t)=\Ordo(|t|^{-\frac14-\epsilon})$ and $g(t)=\Ordo(|t|^{-\frac14-\epsilon})$
hold as $|t|\to+\infty$, the stable solutions $\Sop_-[f]$ and $\Sop_\vert[g]$
are well-defined continuous functions in the whole plane $\R^2$.
This may be seen, e.g., by applying the above asymptotic formula.
Moreover, we
may verify that each function $u=\Sop_-[f]$ and $u=\Sop_\vert[g]$ solves the
Klein-Gordon equation $u''_{xy}+u=0$ in the sense of distribution theory,
with critical line data $\Sop_-[f](x,0)=f(x)$ and $\Sop_\vert[g](0,y)=g(y)$.  
\end{rem}

If we knew that the stable solutions enjoy appropriate growth bounds, we may
use the theorems of the preceding subsection to show that they are unique
as solutions of moderate growth. 

\begin{prop}\label{prop:stable-growth}
If $f,g:\R\to\C$ are continuous with $f(t)=\Ordo(|t|^{-\frac14-\epsilon})$ and
$g(t)=\Ordo(|t|^{-\frac14-\epsilon})$ for some constant $\epsilon>0$, then
the stable solutions enjoy the growth bounds
\[
|\Sop_-[f](x,y)|=
\Ordo\big((1+|x|)^{-\frac14-\epsilon}+\epsilon^{-1}|y|^{\frac14}\big)
\]
and
\[
|\Sop_\vert[g](x,y)|=
\Ordo\big((1+|y|)^{-\frac14-\epsilon}+\epsilon^{-1}|x|^{\frac14}\big)
\]
with constant that are uniform in the plane as $\epsilon\to0^+$. Moreover,
$\Sop_-[f]$ and $\Sop_\vert[g]$ are the unique global solutions to the
Klein-Gordon equation with these growth bounds and data $\Sop_-[f](\cdot,0)=f$
$\Sop_\vert[g](0,\cdot)=g$. 
\end{prop}

The proof is given in the appendix (Section~\ref{sec:appendix-1}).

\begin{rem}
If $f(0)=g(0)=0$ and  $f(t)=\Ordo(|t|^{-\frac14-\epsilon})$ and
$g(t)=\Ordo(|t|^{-\frac14-\epsilon})$ for some constant $\epsilon>0$, then 
\[
\Sop_-[f](0,y)=\begin{cases}
-\int_{-\infty}^0 J_{0,1}(-t,y)f(t)\,\diff t\quad\text{if}\,\,\, y\ge0,
\\
\int_{0}^{+\infty} J_{0,1}(-t,y)f(t)\,\diff t\quad\text{if}\,\,\, y<0,
\end{cases}
\]
while
\[
\Sop_\vert[g](x,0)=\begin{cases}
-\int_{-\infty}^0 J_{1,0}(x,-t)g(t)\,\diff t\quad\text{if}\,\,\, x\ge0,
\\
\int_{0}^{+\infty} J_{1,0}(x,-t)g(t)\,\diff t\quad\text{if}\,\,\, x<0,
\end{cases}
\]
The mappings $f\mapsto \Sop_-[f](0,\cdot)$ and $g\mapsto\Sop_\vert[g](\cdot,0)$
may then be viewed as versions of the classical Hankel transform. In fact,
the mappings $f\mapsto \Sop_-[f](0,\cdot)$ and
$g\mapsto\Sop_\vert[g](\cdot,0)$ reduce to Hankel transforms of order $1$ in
squared radial coordinates. On the corresponding semi-axes, these
transformations are mutually inverse, in the sense of Hankel inversion,
with the appropriate interpretation of the integrals. Also, note that this
interpretation allows for a weighted $L^2$ Plancherel identity on each
semi-axis. A case in point is the Hankel-Plancherel identity
\[
\int_{-\infty}^{0}|\Sop_-[f](0,y)|^2\frac{\diff y}{|y|}
=\int_{0}^{+\infty}|f(t)|^2\frac{\diff t}{t}.  
\]
\end{rem}

\subsection{A nonlinear extension: the sinh-Gordon equation}
\label{subsec:sg-laplace}
A natural nonlinear extension of the present analysis concerns the
sinh-Gordon equation
\[
u''_{xy}+\sinh u=0.
\]
Its linearization at the zero solution is precisely the Klein-Gordon
equation considered here. The sinh-Gordon equation is also completely
integrable, which makes it a particularly interesting candidate for a
nonlinear analogue of our theory.

The reduction to an evolution equation survives without linearity.
Indeed, suppose that a sufficiently regular solution satisfies
$u(0,y)=0$ along a characteristic semi-axis. Integration in the
$x$-variable then yields
\[
\partial_yu(x,y)=-\int_0^x\sinh u(s,y)\,\diff s.
\]
Writing $t=-y$, $v(x,t)=u(x,-t)$, and
\[
\Jop f(x)=\int_0^x f(s)\,\diff s,
\]
we obtain the nonlinear Volterra evolution
\[
\partial_tv=\Jop(\sinh v).
\]
On a finite $x$-interval this is a locally well-posed evolution
in the space of continuous functions. Its linearization at zero is
$\partial_tv=\Jop v$, the evolution underlying the linear theory.

It would be interesting to determine whether the nonlinear evolution
admits distinguished classes of globally defined stable solutions and
whether corresponding uniqueness principles hold. The appropriate
growth assumptions may differ drastically from those in the present
paper. In particular, the exponential growth of $\sinh u$ for large
real $u$ introduces mechanisms absent from the linear equation.
The integrable structure of the sinh-Gordon equation may provide additional tools,
but no nonlinear counterpart of our sharp growth thresholds is claimed
here.

\subsection{Further questions}

Another natural direction concerns the Klein-Gordon equation in
higher-dimensional Minkowski space. For radial solutions in three
spatial dimensions, multiplication by the radial variable, which is
proportional to \(x-y\) in characteristic coordinates, reduces the equation
to the one-dimensional Klein-Gordon equation. In the spacelike
quarter-plane considered here, this reduction takes place entirely away
from the spatial origin. It would be interesting to investigate
whether our sharp uniqueness results admit corresponding radial formulations
in three spatial dimensions. In other dimensions, the radial reduction
introduces an inverse-square potential.

\subsection{Acknowledgements}
As mentioned earlier, this work continues early insights by my Ukrainian
colleague Andrew Bakan, dating back many years. We definitely appreciate
his contribution to the topic.

\section{Classical properties of the Klein-Gordon equation}
\label{sec:classical-KG}

\subsection{Riemann solution and characteristic geometry}
The Klein-Gordon equation describes the wavefunction of a relativistic
spinless boson with positive rest mass. In one spatial and one temporal
dimension, it is advantageous to switch to using the characteristic directions
to define alternative coordinates, which after rescaling gives us the equation
$u\,\!_{xy}^{\prime\prime}+u=0$
where $u(x,y)$ is the wavefunction in question. The characteristic directions
are now the lines parallel to the $x$-axis as well as those parallel to
the $y$-axis. 
In this context, the corresponding wave equation reads $u''_{xy}=0$ with general
solution of the form $u(x,y)=f(x)+g(y)$. Here, in principle, the solution
exhibits no extra smoothness and if we look for distributional solutions,
then $f,g$ should be thought of as arbitrary distributions of one variable.
We shall be concerned mainly with \emph{continuous solutions} $u$ to the wave
equation as well as to the Klein-Gordon equation. In the context of the wave
equation with general solution $u(x,y)=f(x)+g(y)$, the univariate functions
$f,g$ are then continuous as well.
The \emph{Lorentz group} describes the linear transformations of the
$(x,y)$-plane which preserve the shape of the Klein-Gordon equation. In our
simplified $1+1$ dimensional context, it consists of the $2\times2$ matrices
\[
A=
\begin{pmatrix}
\lambda & 0\\
0&\lambda^{-1}
\end{pmatrix}\quad\text{or}\quad
\begin{pmatrix}
0 & \lambda\\
\lambda^{-1}&0
\end{pmatrix},
\]
where $\lambda\in\R$, $\lambda\ne0$. In addition to such linear Lorentz
transformations, the Klein-Gordon equation is also invariant under all
translations $(x,y)\mapsto (x+a,y+b)$, for $(a,b)\in\R^2$, which entails that
the equation is invariant under the action of the semidirect sum of the
translation group and the Lorentz group, called the \emph{Poincar\'e group}. 

It is natural to consider the \emph{Darboux-Goursat problem} for such
hyperbolic equations. For the Klein-Gordon equation, the Darboux-Goursat
problem associated with the given characteristic axes reads
\begin{equation}
\begin{cases}
u''_{xy}+u=0,\\
u(x,0)=f(x),
\quad u(0,y)=g(y).  
\end{cases}    
\label{eq:KG+Goursat-1}
\end{equation}
Here, given that we ask for a continuous solution $u$, we should require that
$f,g$ be continuous with $f(0)=g(0)$. 
As a side note, we can mention that a discretized version of the 
Darboux-Goursat problem \eqref{eq:KG+Goursat-1} was considered recently in 
\cite{BHM}, for a restricted class of bounded solutions $u$. Further background 
material pertaining to uniqueness issues for such discretized Darboux-Goursat 
problems can be found in \cite{HM}, \cite{HM2}, \cite{HM3}, \cite{Hed}, as well 
as the connection with Fourier interpolation for radial functions in
\cite{BHMRV}. A further connection with dynamical systems, through
Perron--Frobenius operators, is developed by Canto-Mart\'in,
Hedenmalm, and Montes-Rodr\'iguez \cite{CHM}. 
The classical Riemann method for second-order hyperbolic equations, including characteristic initial-value problems, is treated systematically in Courant and Hilbert \cite{CourantHilbert}. The distinction between the two-characteristic Darboux--Goursat problem and characteristic Cauchy uniqueness is important here.\par
Equations of the type \eqref{eq:KG+Goursat-1} were of interest to Bernhard 
Riemann, and for this reason the solution is associated with his name. Before
we come to that, we need to introduce certain Bessel-type functions.
For $a=0,1,2,\ldots$, the \emph{bivariate Bessel function} 
\[
J_{0,a}(y,x):=J_{a,0}(x,y):=\sum_{j=0}^{+\infty}\frac{(-1)^j}{(j+a)!j!}\,x^{j+a}y^{j}
=\bigg(\frac{x}{y}\bigg)^{a/2}J_a(2\sqrt{xy})
\]
is an entire function of two variables, where on the right-hand side the
function $J_a$ denotes the usual Bessel function. We quickly verify that
\begin{equation*}
\partial_x\partial_y J_{a,0}(x,y)+J_{a,0}(x,y)=0,
\end{equation*}
so that $J_{a,0}$ constitutes a solution to the Klein-Gordon equation.
Moreover, since we have agreed that $J_{a,0}(x,y)=J_{0,a}(y,x)$, the
function $J_{0,a}(x,y)$ solves the Klein-Gordon equation as well.  
Note that differentiation has the effect
\[
\partial_x J_{a,0}(x,y)=J_{a-1,0}(x,y),\qquad a\ge1,
\]
while
\[
\partial _x J_{0,0}(x,y)=-J_{0,1}(x,y),
\]
and, similarly, we have
\[
\partial_y J_{a,0}(x,y)=-J_{a+1,0}(x,y),\qquad a\in\Z_{\ge0}.
\]

Finally, we note that restriction to the lines $x=0$ and $y=0$ gives 
\[
J_{0,0}(0,y)=1,\quad J_{a,0}(0,y)=0\,\,\,\text{if}\,\,\,a\ge1,
\]
and
\[
J_{a,0}(x,0)=\frac{x^a}{a!}.
\]

\begin{defn}
\emph{Riemann's solution} to the Darboux-Goursat problem is given by the 
expression
\begin{multline}
\Rop[f,g](x,y):=f(x)+g(y)-f(0)J_{0,0}(x,y)-\int_0^x J_{0,1}(x-s,y)f(s)\,\diff s
\\
-\int_0^y J_{1,0}(x,y-t)g(t)\,\diff t.  
\end{multline}  
\end{defn}

We now make some observations which are classical.
To begin with, we note that
\begin{equation}
\Rop[f,g](x,0)=f(x)+g(0)-f(0)J_{0,0}(x,0)-\int_0^x J_{0,1}(x-s,0)f(s)\,\diff s
=f(x)  
\label{eq:Goursat-1}
\end{equation}
since $f(0)=g(0)$, $J_{0,0}(x,0)=1$, and $J_{0,1}(x-s,0)=0$. By interchanging 
the roles of $x$ and $y$ in the calculation, we find likewise that
\begin{equation}
\Rop[f,g](0,y)=f(0)+g(y)-f(0)J_{0,0}(0,y)-\int_0^y J_{1,0}(0,y-t)g(t)\,\diff t
=g(y).
\label{eq:Goursat-2}
\end{equation}
It now follows that $u=\Rop[f,g]$ has the correct boundary data along the
axes $x=0$ and $y=0$.

\begin{thm}
{\rm (Riemann)}
Provided that $f,g$ are continuous univariate functions with $f(0)=g(0)$, the
function $u(x,y)=\Rop[f,g](x,y)$ is a continuous function of two variables,
which solves the Darboux-Goursat problem \eqref{eq:KG+Goursat-1} in the sense
of distribution theory.
\label{thm:classical-1}
\end{thm}

\begin{proof}[Proof of Theorem \ref{thm:classical-1}]
It is an easy consequence of the formula defining $u=\Rop[f,g]$ that $u$
gets to be continuous. Moreover, the values along the axis are as required,
in accordance with the calculations
\eqref{eq:Goursat-1} and \eqref{eq:Goursat-2}.

It remains to check that $u=\Rop[f,g]$ solves the Klein-Gordon equation
$u''_{xy}+u=0$ in the sense of distribution theory. To this end, the
first step is to realize that
\[
\partial_x\partial_y(f(x)+g(y))=0
\]
holds in the sense of distribution theory. It now follows from the definition
of Riemann's solution that in the sense of distribution theory,
\begin{multline}
\partial_x\partial_y u(x,y)=\partial_x\partial_y(f(x)+g(y))
-f(0)\partial_x\partial_y J_{0,0}(x,y)
\\
-\partial_x\partial_y\int_0^x J_{0,1}(x-s,y)f(s)\,\diff s
-\partial_x\partial_y\int_0^y J_{1,0}(x,y-t)g(t)\,\diff t
\\
=f(0)J_{0,0}(x,y)-\partial_x\partial_y\int_0^x J_{0,1}(x-s,y)f(s)\,\diff s
-\partial_x\partial_y\int_0^y J_{1,0}(x,y-t)g(t)\,\diff t,
\label{eq:step1}
\end{multline}
where we applied the fact that the function $J_{0,0}$ solves the Klein-Gordon
equation itself. Next, by Calculus, we have that
\begin{multline}
\partial_x\int_0^x J_{0,1}(x-s,y)f(s)\,\diff s=J_{0,1}(0,y)f(x)+\int_0^x
\partial_x J_{0,1}(x-s,y) f(s)\,\diff s
\\
=f(x)y+\int_0^x
\partial_x J_{0,1}(x-s,y) f(s)\,\diff s,
\end{multline}
and, consequently,
\begin{multline}
\partial_x\partial_y\int_0^x J_{0,1}(x-s,y)f(s)\,\diff s=
\\
=f(x)+\int_0^x
\partial_x\partial_y J_{0,1}(x-s,y) f(s)\,\diff s
=f(x)-\int_0^x J_{0,1}(x-s,y) f(s)\,\diff s.
\label{eq:step2}
\end{multline}
The analogous calculation switching the roles of $x$ and $y$ gives that
\begin{equation}
\partial_x\partial_y\int_0^y J_{1,0}(x,y-t)g(t)\,\diff t
=g(y)-\int_0^y J_{1,0}(x,y-t) g(t)\,\diff t.
\label{eq:step3}
\end{equation}
It now follows by a combination of \eqref{eq:step1}, \eqref{eq:step2}, an
\eqref{eq:step3} that
\begin{multline}
\partial_x\partial_y u(x,y)=f(0)J_{0,0}(x,y)-f(x)-g(y)+\int_0^x J_{1,0}(y,x-s)f(s)
\,\diff s
\\
+\int_0^y J_{1,0}(x,y-t)g(t)\,\diff t=  -u(x,y).
\end{multline}
This means that $u=\Rop[f,g]$ solves the Klein-Gordon equation $u''_{xy}+u=0$
in the sense of distribution theory, which completes the proof.
\end{proof}

Riemann's contribution to the Darboux-Goursat problem with data along 
characteristics can be found in his collected works \cite{Riemann} (the paper 
\emph{Ueber die Fortplantzung ebener Luftwellen von endlicher Schwingungsweite} 
is the most relevant), and for this reason, the method
is called the \emph{Riemann method}.

It remains to determine that Riemann's solution is the unique solution to the
Darboux-Goursat problem \eqref{eq:KG+Goursat-1}. To this end,
it is an important observation that continuous distributional solutions to
the Klein-Gordon equation $u''_{xy}+u=0$ are characterized by the
\emph{quadrature identity}
\begin{equation}
\int_a^{a'}\int_{b}^{b'}u(s,t)\,\diff t\diff s=u(a',b)+u(a,b')-u(a',b')-u(a,b).
\label{eq:quadrature-1}  
\end{equation}
If we apply \eqref{eq:quadrature-1} to $a'=x,b'=y,a=b=0$, we may view the
equation as a way to represent $u(x,y)$ in terms of the values $u(x,0)$, $u(0,y)$,
$u(0,0)$, together with a Volterra-type integral:
\begin{equation}
u(x,y)=u(x,0)+u(0,y)-u(0,0)-\int_0^{x}\int_{0}^{y}u(s,t)\,\diff t\diff s.
\label{eq:quadrature-1.1}  
\end{equation}
Based on this representation, the Picard method applies to give uniqueness. 

\begin{thm}
{\rm (Picard's method)}
Suppose $u(x,y)$ is a continuous solution of the Darboux-Goursat problem
\eqref{eq:KG+Goursat-1}, where the Klein-Gordon equation is understood in the
sense of distribution theory, and the univariate functions $f,g$ are continuous
with $f(0)=g(0)$. Then $u$ coincides with Riemann's solution: $u=\Rop[f,g]$. 
\label{thm:classical-2}
\end{thm}

\begin{proof}[Proof of Theorem \ref{thm:classical-2}]
By considering the difference $v=u-\Rop[f,g]$, we have a continuous solution
to the Klein-Gordon equation $v''_{xy}+v=0$ with $v(x,0)=0$ and $v(0,y)=0$.    
By the quadrature identity \eqref{eq:quadrature-1.1} with $v$ in place of $u$,
\begin{equation}
v(x,y)=-\int_0^{x}\int_{0}^{y}v(s,t)\,\diff t\diff s.
\label{eq:quadrature-2}  
\end{equation}
Let $Q$ denote the rectangle $a\le x\le a'$, $b\le y\le b'$, where the
parameters are chosen such that the origin $(0,0)$ is in $Q$. In view of
\eqref{eq:quadrature-2}, we have
\begin{equation}
  |v(x,y)|\le|xy|\,\|v\|_{L^\infty(Q)},\qquad (x,y)\in Q, 
\label{eq:quadrature-3}  
\end{equation}
so that if the rectangle $Q$ is not too big, more precisely,
\[
A:=\max\{|a|,a'\}\,\max\{|b|,b'\}<1,
\]
we find that
\[
\|v\|_{L^\infty(Q)}\le A\|v\|_{L^\infty(Q)}, 
\]
which forces $\|v\|_{L^\infty(Q)}=0$, that is, $v=0$ holds on $Q$. But then we
can do a horizontal translation $(x,y)\mapsto (x+\beta,y)$
(for a small $\beta$, positive or negative) and apply the same argument
again after the translation. Iteratively, we find that $v=0$ holds in the whole
strip $-\infty<x<+\infty$, $b\le y\le b'$. After that we apply instead vertical
translations and obtain that $v=0$ throughout. 
\end{proof}

The classical theory of the characteristic initial-value problem for
hyperbolic equations is discussed, for instance, by Courant and Hilbert
\cite{CourantHilbert} and John \cite{JohnPDE}. The formulation allows continuous distributional solutions.

\subsection{Finite speed propagation and uniqueness}
The Klein-Gordon equation is relativistic in the sense that the speed of light
$c$ is a finite positive constant, which in our modelling is set to $c=1$.
Let us see what this means in terms of Riemann's solution $u=\Rop[f,g]$.
So, we suppose that the continuous data $f,g$ in the Darboux-Goursat problem
\eqref{eq:KG+Goursat-1} vanishes initially:
\[
f(t)=0 \,\,\,\text{for all}\,\,\,t\in[a,a'],\quad
g(t)=0\,\,\,\text{for all}\,\,\, t\in[b,b'],  
\]
for some reals $a\le0\le a'$ and $b\le0\le b'$. It then follows by inspection
of Riemann's solution that
\[
u(x,y)=\Rop[f,g](x,y)=0,\qquad (x,y)\in[a,a']\times[b,b'].
\]
This has the interpretation that information travels at finite speed, and that
the values on the rectangle $[a,a']\times[b,b']$ are fully determined by the
values on the union of the characteristic segments $[a,a']\times\{0\}$ and
$\{0\}\times[b,b']$. Nevertheless, this presupposes that $u$ is a global
continuous solution to the Klein-Gordon equation. What if $u$ is just a
continuous function on the rectangle $[a,a']\times[b,b']$ that 
that solves the Klein-Gordon equation in the interior in the sense of
distribution theory? The Picard argument still applies to give uniqueness:

\begin{thm}
{\rm (Picard's method)}
Suppose $a\le 0\le a'$ and $b\le 0\le b'$, and that $u$ is a continuous
function on the closed rectangle $[a,a']\times[b,b']$ which solves the
Klein-Gordon equation in the interior of the rectangle in the sense of
distribution theory. If
\[
u(x,0)=u(0,y)=0,\qquad x\in[a,a'],\,\,\,y\in[b,b'],
\]
then $u=0$ holds on $[a,a']\times[b,b']$. 
\label{thm:Picard-2}
\end{thm}

\begin{proof}[Proof of Theorem \ref{thm:Picard-2}]
As in the proof of Theorem \ref{thm:classical-2} , we have, by the quadrature 
identity \eqref{eq:quadrature-1.1},
\begin{equation}
u(x,y)=-\int_0^{x}\int_{0}^{y}u(s,t)\,\diff t\diff s, \qquad a\le x\le a',\,\,
b\le y\le b',
\label{eq:quadrature-2.01}  
\end{equation}
which puts us in the same setting as in the proof of Theorem 
\ref{thm:classical-2}. The conclusion that $u=0$ on $[a,a']\times[b,b']$
follows in an analogous fashion. 
\end{proof}

The proof follows the same argument as that of Theorem \ref{thm:classical-2}, as shown above.

\subsection{Poincar\'e group invariance of the Klein-Gordon equation}
We formalize the invariance of the Klein-Gordon equation under the Poincar\'e
group in a proposition.

\begin{prop}
For a real parameter $\lambda\ne0$, and two real parameters $a,b$, let $L$
and $M$ denote the affine transformations
$L(x,y)=(\lambda x+a,\lambda^{-1}y+b)$ and 
$M(x,y)=(\lambda y+a,\lambda^{-1} x+b)$.
If $u$ solves the Klein-Gordon equation $u''_{xy}+u=0$ on a domain
$\Omega\subset\R^2$, then so do $v=u\circ L$ and $w=u\circ M$ on the
corresponding domains $L^{-1}(\Omega)$ and $M^{-1}(\Omega)$, respectively. 
\label{prop:Poincare-1}
\end{prop}

This result is well-known and we omit the proof. 

\begin{rem}
The assertion of the proposition holds not only for classical solutions
but for (weak) solutions in the sense of distribution theory as well.
\end{rem}

\subsection{Gluing solutions of the Klein-Gordon equation along 
characteristics}

Suppose we are given a bounded convex domain $\Omega\in\R^2$, and a 
line $L\subset \R^2$  which is such that $\Omega\cap L$ has positive length.
Suppose moreover $u:\Omega\to\C$ is continuous and harmonic in 
$\Omega\setminus L$. Then it is not necessarily true that $u$ extends 
to be harmonic in $\Omega$. What can prevent harmonic extension is the 
appearance of a \emph{Neumann jump} along $L$, that is, the first-order derivative
along the normal direction has a jump across $L$.
As we replace the Laplacian (which has order $2$) with Cauchy-Riemann 
operator $\bar\partial$, which has order $1$, this problem does not occur, 
and we automatically get analytic continuation along any given linear segment 
$L\cap\Omega$ if the given function is 
continuous on $\Omega$ and holomorphic in $\Omega\setminus L$.  
As the Klein-Gordon equation $u''_{xy}+u=0$ has order $2$, the 
appearance of a nontrivial Neumann jump along the line $L$ may influence 
the extendability across the line as a solution. 
This, however, is not the case if the line is characteristic.

\begin{thm}
Suppose $\Omega\subset\R^2$ is bounded and convex and $L$ is a line
which is characteristic, that is, parallel to either the $x$-axis or the $y$-axis.
If $u:\Omega\to\C$ is continuous and solves the Klein-Gordon equation
$u''_{xy}+u=0$ in $\Omega\setminus L$ in the sense of distribution theory, 
then $u$ extends to a solution to the Klein-Gordon equation on all of 
$\Omega$.
\label{thm:glue-1}
\end{thm}

\begin{proof}[Proof of Theorem \ref{thm:glue-1}]
The assertion is trivial unless $\Omega\cap L\neq\emptyset$, so
we will make this assumption.
In view of the stated characterization of continuous solutions to the 
Klein-Gordon equation in terms of the quadrature identity 
\eqref{eq:quadrature-1} for any rectangle $[a,a']\times[b,b']$ contained
in $\Omega$, we need to check that 
\begin{equation}
\int_{a}^{a'}\int_{b}^{b'}u(s,t)\diff t\diff s=u(a,b')+u(a',b)-u(a,b)-u(a',b').
\label{eq:quadrature-2.1}
\end{equation}
By Lorentz group invariance,  it is sufficient to handle the instance when
$L$ is a horizontal line $y=b''$. The distributional interpretation of 
continuous solutions to the Klein-Gordon equation gives that we know 
that \eqref{eq:quadrature-2.1} holds for all 
$[a,a']\times[b,b']\subset\Omega\setminus L$, that is, when 
$[a,a']\times[b,b']\subset\Omega$ and $b''\notin[b,b']$. In a next step,
we may use the continuity of $u$ to allow $b$ or $b'$ to approach the
value $b''$, which establishes \eqref{eq:quadrature-2.1} for
$[a,a']\times[b,b']\subset\Omega$ and $b''\notin]b,b'[$. Finally, when
$b''\in]b,b'[$, we split $[b,b']=[b,b'']\cup[b'',b']$,  and since,
by the argument we already carried out, 
\begin{equation}
\int_{a}^{a'}\int_{b}^{b''}u(s,t)\diff t\diff s=u(a,b'')+u(a',b)-u(a,b)-u(a',b'')
\label{eq:quadrature-2.2}
\end{equation}
and likewise
\begin{equation}
\int_{a}^{a'}\int_{b''}^{b'}u(s,t)\diff t\diff s=u(a,b')+u(a',b'')-u(a,b'')-u(a',b').
\label{eq:quadrature-2.3}
\end{equation}
Adding up these two identities, it follows that \eqref{eq:quadrature-2.1}
holds in this instance as well.
The assertion of the theorem now follows from the general quadrature 
identity \eqref{eq:quadrature-2.1}.
\end{proof}

\begin{rem}
The gluing property is a consequence of the characteristic nature of the
interface.  For a continuous function which solves $u''_{xy}+u=0$
distributionally on both sides of a horizontal or vertical line, the
distributional mixed derivative has no interface-supported contribution:
the possible jump term from differentiating across the interface
vanishes by continuity.  By contrast, across a noncharacteristic line
continuity alone does not in general rule out a singular interface
term; an appropriate derivative matching condition is then needed.
For general background on distributional solutions of constant-coefficient
partial differential equations, see H\"ormander~\cite{HormanderII}.
For a related, but substantially different, characteristic gluing
problem for the wave equation on null hypersurfaces, involving
conservation-law obstructions, see Aretakis~\cite{AretakisGluing}.
We retain the elementary proof here; neither reference is invoked
as a source for the precise statement above.
\end{rem}

\subsection{Splitting Riemann's solution into unilateral wave solutions}
\label{ss:splitting}
Riemann's solution $\Rop[f,g]$ to the Darboux-Goursat problem
\eqref{eq:KG+Goursat-1} may be split into three components.

\begin{prop}
If $u=\Rop[f,g]$ is Riemann's solution to the Darboux-Goursat problem
\eqref{eq:KG+Goursat-1} with continuous data $f,g$, subject to the
compatibility condition $f(0)=g(0)$, then $u$ splits as
\[
u=u_{0}+u_{\mathrm{horz}}+u_{\mathrm{vert}},
\]
where
\[
u_0(x,y)=f(0)J_{0,0}(x,y)=f(0)\Rop[1,1](x,y)
\]
and 
\[
u_{\mathrm{horz}}=\Rop[f_0,0],\quad u_{\mathrm{vert}}=\Rop[0,g_0],
\]
with $f_0(t)=f(t)-f(0)$ and $g_0(t)=g(t)-g(0)$. Here, each component $u_0$,
$u_{\mathrm{horz}}$, $u_{\mathrm{vert}}$ is continuous and solves the Klein-Gordon
equation $u''_{xy}+u=0$ individually. 
\label{prop:split-1}
\end{prop}

This assertion is immediate from Theorems \ref{thm:classical-1} and 
\ref{thm:classical-2} 
by slightly rearranging the terms in Riemann's solution. 
Alternatively, we can observe that the expressed solution $u$ solves the Klein-Gordon
equation $u''_{xy}+u=0$ in the sense of distribution theory since each of the terms 
$u_0,u_{\mathrm{horz}},u_{\mathrm{vert}}$ does, and that the values along the 
$x$-axis and the $y$-axis coincide with $f$ and $g$, respectively. Theorem 
\ref{thm:classical-2} then shows that $u=u_0+u_{\mathrm{horz}}+u_{\mathrm{vert}}$
must coincide with $\Rop[f,g]$. 
Proposition \ref{prop:split-1} motivates the following terminology.

\begin{defn}
The wave $u_{\mathrm{horz}}=\Rop[f_0,0]$ is called a \emph{horizontal unilateral
wave solution}, whereas the wave $u_{\mathrm{vert}}=\Rop[0,g_0]$ is called a
\emph{vertical unilateral wave solution}. Together, we refer to the waves of
the form $u_{\mathrm{horz}}$ or $u_{\mathrm{vert}}$ as
\emph{unilateral wave solutions}.
\end{defn}

The following proposition characterizes such unilateral wave solutions. 

\begin{prop}
Suppose $u:\R^2\to\C$ is continuous and solves the Klein-Gordon equation
$u''_{xy}+u=0$ in the sense of distribution theory. Then

\noindent{\rm(a)} $u=u_{\mathrm{horz}}$ holds if and only if $u(0,y)=0$
for all $y\in\R$, and

\noindent{\rm(b)} $u=u_{\mathrm{vert}}$ holds if and only if $u(x,0)=0$
for all $x\in\R$.
\label{prop:unilat-1}
\end{prop}

In view of the properties of Riemann's solution $u=\Rop[f,g]$ outlined in
Theorems \ref{thm:classical-1} and \ref{thm:classical-2}, the assertion of
the proposition is immediate from the definition of unilateral wave solutions, 
horizontal and vertical. 

In this paper, we study unilateral wave solutions under growth restrictions.
By Lorentz invariance, it suffices to consider horizontal unilateral solutions,
and we shall do so throughout.

\subsection{An illustrative unilateral solution and its growth}
\label{ss:growth-1}

We consider the Riemann solution $u_1=\Rop[f_1,0]$, where the continuous
function $f_1$ is given by
\[
f_1(x):=
\begin{cases}
0,\qquad x\le 0,
\\
x,\qquad 0<x<+\infty.
\\
1,\qquad x\ge1.  
\end{cases}
\]
Then $u_1$ is a horizontal unilateral wave solution, expressed by
\[
u_1(x,y)=\Rop[f_1,0](x,y)=f_1(x)-\int_0^x J_{0,1}(x-s,y)f_1(s)\,\diff s.
\]
Since we study $y\in\R_{\le0}$, it is convenient to flip $y\mapsto-y$:
\begin{multline}
u_1(x,-y)=\Rop[f_1,0](x,-y)=
f_1(x)-\int_0^x J_{0,1}(x-s,-y)f_1(s)\,\diff s
\\
=f_1(x)-\int_0^x \sum_{j=0}^{+\infty}\frac{(-1)^j}{j!(j+1)!}(x-s)^j(-y)^{j+1}
f_1(s)\,\diff s
\\
=f_1(x)+\sum_{j=0}^{+\infty}\frac{y^{j+1}}{j!(j+1)!}\int_0^x (x-s)^j
f_1(s)\,\diff s.
\end{multline}
Since $0\le f_1(x)\le1$ holds for $x\in\R_{\ge0}$, it now follows that
\begin{multline}
0\le u_1(x,-y)\le 1+\sum_{j=0}^{+\infty}\frac{y^{j+1}}{j!(j+1)!}\int_0^x (x-s)^j
\diff s
\\
=1+\sum_{j=0}^{+\infty}\frac{(xy)^{j+1}}{[(j+1)!]^2}=J_{0,0}(x,-y),
\qquad (x,y)\in\R_{\ge0}\times\R_{\ge0}. 
\label{eq:u1-est-1}
\end{multline}
Next, the function $J_{0,0}$ 
may be expressed in terms of the standard Bessel function:
\[
J_{0,0}(x,-y)=J_0(2\sqrt{-xy}),
\]
Classical asymptotics of the Bessel function gives that as $xy\to+\infty$
while $(x,-y)$ is in the spacelike quarter-plane where $x,y>0$, we
have that
\begin{equation}
J_{0,0}(x,-y)=(1+\ordo(1))\frac{1}{2\sqrt{\pi}}\,(xy)^{-1/4}
\exp\big(2\sqrt{xy}\big),
\label{eq:uinfty-2}
\end{equation}
whereas if $xy$ is kept bounded, we have
\[
J_{0,0}(x,-y)=\Ordo(1).
\]

\begin{proof}[Proof of Theorem \ref{thm:main-1-1}]
By Young's inequality for products, in fact, in this case,
just by completing squares, we have that
\[
2\sqrt{xy}\le \theta_1 x+\theta_1^{-1}y\le \theta_1 x+\theta_2 y,
\qquad x,y\in\R_{\ge0},  
\]
given that $\theta_1\theta_2\ge1$ entails that $\theta_2\ge\theta_1^{-1}$.
It now follows from the above control on $J_{0,0}(x,-y)$ combined with the
inequalities in \eqref{eq:u1-est-1} that
\[
u_1(x,-y)=\Ordo(\exp(\theta_1 x+\theta_2 y)),\qquad (x,y)\in
\R_{\ge0}\times\R_{\ge0},
\]
holds uniformly. Moreover, since $u_1(x,0)=f_1(x)$ holds for $x\in\R$,
it follows that $u_1$ is a nontrivial solution to the Klein-Gordon equation
which in particular does not vanish identically on
$\R_{\ge0}\times\R_{\le0}$ while it enjoys the claimed growth estimate.
\end{proof}

\section{Uniqueness in the exponential growth case $q=1$}
\label{sec:horizontal-unilateral}

\subsection{Uniqueness in the exponential growth case $q=1$}
\label{ss:q=1}

We turn to the proof of the simplest instance of the Liouville-type phenomenon
for the Klein-Gordon equation.

\begin{proof}[Proof of Theorem \ref{thm:main-1-0}]
By Remark \ref{rem:observation-translation}, we may assume that
$0\in\calY$.
We let $f(x):=u(x,0)$ for $x\in\R_{\ge0}$, and observe that by continuity,
we must have $f(0)=0$. Next, Riemann's solution $\Rop[f,0]$ is well-defined in
the half-plane $\R_{\ge0}\times\R$, and vanishes on the half-line $\{0\}\times
\R$. Then the difference function $u-\Rop[f,0]$ is continuous, solves
the Klein-Gordon equation in the sense of distribution theory, and vanishes
along the two characteristic semi-axes $\R_{\ge0}\times\{0\}$ and
$\{0\}\times\R_{\le0}$. By Theorem \ref{thm:Picard-2}, then, $u-\Rop[f,0]=0$
on $\R_{\ge0}\times\R_{\le0}$, so that $u=\Rop[f,0]$ holds there, so that we
may think of $u=u_{\mathrm{horz}}$ as a unilateral horizontal wave solution.
The given growth control on $u=\Rop[f,0]$ says that
\[
|u(x,y)|\le C_0\exp(\theta_1 x+\theta_2 |y|),\qquad (x,y)\in\R_{\ge0}
\times\mathcal{Y},
\]
for some positive constant $C_0$. In particular, since $0\in\mathcal{Y}$,
we have
\[
|f(x)|=|u(x,0)|\le C_0\exp(\theta_1 x),\qquad x\in\R_{\ge0}.
\]
The Laplace transform
\[
\Lop[f](\zeta)=\int_0^{+\infty}\e^{-t\zeta}f(t)\,\diff t,\qquad \re\zeta>\theta_1,
\]
is then holomorphic in the indicated half-plane, with the growth bound
\[
|\Lop[f](\zeta)|\le\frac{C_0}{\re\zeta-\theta_1},\qquad
\re\zeta>\theta_1.  
\]
Given the growth control on the solution $u=u_{\mathrm{horz}}=\Rop[f,0]$,
we find analogously that
\[
|\Lop[u(\cdot,y)](\zeta)|\le\frac{C_0\exp(\theta_2|y|)}{\re\zeta-\theta_1},
\qquad\re\zeta>\theta_1,\,\,\,y\in \mathcal{Y}.  
\]
Finally, in view of the evolution equation for the Laplace transform
\eqref{eq:Ropf0-2}, we know that 
\[
|\e^{-y/\zeta}\Lop[f](\zeta)|\le\frac{C_0\exp(\theta_2|y|)}{\re\zeta-\theta_1},
\qquad\re\zeta>\theta_1,\,\,\,y\in \mathcal{Y}, 
\]
and hence
\[
|\Lop[f](\zeta)|\le\frac{C_0}{\re\zeta-\theta_1}\,
\exp\big(\theta_2|y|+\re(y/\zeta)\big),
\qquad\re\zeta>\theta_1.
\]
The left-hand side does not depend on $y\in\mathcal{Y}$, so we are free
to minimize over $y$. Since $y=-|y|$, we find that
\[
|\Lop[f](\zeta)|\le\frac{C_0}{\re\zeta-\theta_1}\,
\inf_{y\in\mathcal{Y}}\,\exp\big(|y|(\theta_2-\re(1/\zeta))\big),
\qquad\re\zeta>\theta_1,\,\,\,y\in \mathcal{Y},
\]
which gives that
\[
|\Lop[f](\zeta)|=0\quad\text{if}\,\,\,\re\zeta>\theta_1\,\,\,
\text{and}\,\,\,\re(1/\zeta)>\theta_2.   
\]
Given that $\theta_1\theta_2<1$, the indicated region of vanishing is
open and nonempty, so that by the uniqueness theorem for holomorphic functions,
it follows that $\Lop[f]=0$ identically, an hence $f=0$ on $\R_{\ge0}$.
Consequently, the solution $u=\Rop[f,0]$ vanishes in the half-plane $\R_{\ge0}
\times\R$, and, in particular, in the quarter-plane $\R_{\ge0}\times\R_{\le0}$. 
\end{proof}

\section{An auxiliary integral estimate}

\subsection{Estimation of an integral}  

The following estimate is probably known but we have found no suitable
reference.

\begin{lem}
Suppose that $0<q<1$ and $0<\theta,\sigma<+\infty$. Then 
\begin{equation}
\int_0^{+\infty}\exp(\theta\,t^{q}-\sigma t)\, \diff t
\le \frac{1}{\sigma}\big(A(q,\theta) \sigma^{-\frac{q}{2(1-q)}}+B(q)\big)
\,\exp\big(D(q,\theta)\,\sigma^{-\frac{q}{1-q}}\big),
\end{equation}
where 
\[
A(q,\theta)=2^{1-\frac{q}{2}}(q\theta)^{\frac{1}{2(1-q)}}\sqrt{\frac{2\pi}{1-q}}\, ,
\quad B(q)=\frac{4}{1-q},\quad D(q,\theta)=(1-q)\theta (q\theta)^{\frac{q}{1-q}}.
\]
\label{lem:basic-est-1}
\end{lem}

The key estimate to arrive at Lemma \ref{lem:basic-est-1} is the following.

\begin{lem}
Suppose $\psi:\R_{\ge0}\to\R$ is $C^3$-smooth on $\R_{>0}$,
with $\psi''>0$ and $\psi'''<0$ on $\R_{>0}$. If $t_0\in\R_{>0}$ is such
that $\psi'(t_0)=0$, and $0<t_0<t_1<+\infty$, we then have the estimate
\begin{equation}
\int_0^{+\infty}\exp(-\psi(t))\,\diff t\le \e^{-\psi(t_0)}\sqrt{\frac{2\pi}{\psi''(t_1)}}
+\frac{\e^{-\psi(t_1)}}{\psi'(t_1)}.
\end{equation}
\label{lem:basic-est-2}
\end{lem}

\begin{proof}
We begin by splitting the integral:
\begin{equation}
\int_0^{+\infty}\exp(-\psi(t))\,\diff t
=\int_0^{t_1}\exp(-\psi(t))\,\diff t
+\int_{t_1}^{+\infty}\exp(-\psi(t))\,\diff t.
\label{eq:psiint-1}
\end{equation}
Since $\psi''>0$, the function $\psi$ is strictly convex, and consequently,
\begin{equation}
\psi(t)\ge \psi(t_1)+(t-t_1)\psi'(t_1), \qquad x\in \R_{\ge0},
\end{equation}
since convex functions have support lines at each point. Moreover, since
$\psi'(t_0)=0$ and $0<t_0<t_1<+\infty$, and $\psi''>0$, we know that 
$\psi'(t_1)>0$.
As regards the second integral on the right-hand side of \eqref{eq:psiint-1}, 
it now follows that
\begin{equation}
\int_{t_1}^{+\infty}\exp(-\psi(t))\,\diff t
\le \int_{t_1}^{+\infty}\exp\big(-\psi(t_1)-(t-t_1)\psi'(t_1)\big)\diff t
=\frac{\e^{-\psi(t_1)}}{\psi'(t_1)}.
\label{eq:psi-est-2}
\end{equation}
Regarding the first integral on the right-hand side of \eqref{eq:psiint-1},
we use the Lagrange form of the remainder in Taylor's formula:
\begin{equation}
\psi(t)=\psi(t_0)+(t-t_0)\psi'(t_0)+\frac12(t-t_0)^2\psi''(\xi)
=\psi(t_0)+\frac12(t-t_0)^2\psi''(\xi),
\end{equation}
 where $\xi$ is an intermediate point between $t_0$ and $t$. Since
 $\psi'''<0$, the second derivative $\psi''$ is strictly decreasing, so that
 \begin{equation}
\psi(t)=\psi(t_0)+\frac12(t-t_0)^2\psi''(\xi)\ge \psi(t_0)+\frac12(t-t_0)^2
\psi''(t_1),\qquad 0<t\le t_1. 
\end{equation}
 It now follows that
 \begin{multline}
 \int_{0}^{t_1}\exp(-\psi(t))\,\diff t\le\int_{0}^{t_1}\exp\bigg(
 -\psi(t_0)-\frac12(t-t_0)^2\psi''(t_1)\bigg)\diff t
 \\
 \le\e^{-\psi(t_0)}\int_{-\infty}^{+\infty}
 \exp\bigg(-\frac12(t-t_0)^2\psi''(t_1)\bigg)\diff t
 = \e^{-\psi(t_0)}\sqrt{\frac{2\pi}{\psi''(t_1)}}.
 \label{eq:psi-est-3}
 \end{multline}
 The claimed estimate is a consequence of the decomposition
 \eqref{eq:psiint-1} and the estimates \eqref{eq:psi-est-2} and 
 \eqref{eq:psi-est-3}, which completes the proof.
\end{proof}

\begin{proof}[Proof of Lemma \ref{lem:basic-est-1}]
If $\psi(t)=\sigma t-\theta t^q$, we have 
$\psi''(t)=\theta q(1-q)t^{q-2}>0$ and $\psi'''(t)=-\theta
q(1-q)(2-q)t^{q-3}<0$,
and the point $t_0$ with $\psi'(t_0)=0$ is given by
\[
t_0=\bigg(\frac{\theta q}{\sigma}\bigg)^{\frac{1}{1-q}},
\] 
so that 
\[
\psi(t_0)=-\theta(1-q)(q\theta/\sigma)^{\frac{q}{1-q}}.
\]
If $t_1>t_0$, then there is a point $\xi$ between $t_0$ and $t_1$
such that
\[
\psi'(t_1)=\psi'(t_1)-\psi'(t_0)=(t_1-t_0)\psi''(\xi)\ge (t_1-t_0)\psi''(t_1).
\]
We choose $t_1=2t_0$, in which case
\[
\psi''(t_1)=\theta q(1-q)t_1^{q-2}=\theta q(1-q)2^{q-2}t_0^{q-2}=
(1-q)2^{q-2}(q\theta)^{-\frac{1}{1-q}}\sigma^{\frac{2-q}{1-q}},
\]
and hence
\[
\psi'(t_1)\ge(t_1-t_0)\psi''(t_1)=t_0\psi''(t_1)=
(1-q)2^{q-2}\sigma\ge\frac14\,(1-q)\sigma.
\]
Since $\psi(t_1)>\psi(t_0)$, it now follows from Lemma \ref{lem:basic-est-2} 
that 
\begin{multline}
\int_0^{+\infty}\exp(\theta t^q-\sigma t)\,\diff t\le 
\e^{-\psi(t_0)}\Bigg(\sqrt{\frac{2\pi}{\psi''(t_1)}}
+\frac{1}{\psi'(t_1)}\Bigg)
\\
\le\frac{1}{\sigma}\bigg(2^{1-\frac{q}{2}}(q\theta)^{\frac{1}{2(1-q)}}\sqrt{\frac{2\pi}{1-q}}\,
\sigma^{-\frac{q}{2(1-q)}}+\frac{4}{1-q}\bigg) 
\exp\Big((1-q)\theta(q\theta/\sigma)^{\frac{q}{1-q}}\Big),
\end{multline}
which amounts to the claimed estimate of the lemma.
\end{proof}

\section{The instance $0<q<\frac12$ connected with analytic non-quasianalyticity}

\subsection{Background on quasianalyticity and analytic quasianalyticity}
Quasianalytic classes were \`a la mode in the 1920s, with the striking
insights of Denjoy \cite{Denj} and Carleman \cite{Carl}.

If a function on the real line is real-analytic, then, if all its Taylor coefficients
vanish at a single point, the function then must vanish identically. 
This uniqueness property extends to more general $C^\infty$-smooth functions 
on the line, provided we have suitable control on the growth the higher order 
derivatives in a (locally) uniform sense, as given by a weight sequence.  
The precise result is known as the
Denjoy-Carleman theorem, in which the following dichotomy appears. 
Either the class is quasianalytic, in the sense that if all the Taylor coefficients
of the function vanish at a point, then the function must vanish identically,
or the class is non-quasianalytic, meaning that the class contains nontrivial
functions with bounded support. 
As a side note, we observe that Paul Cohen found an elementary proof of the 
Denjoy-Carleman theorem \cite{Coh}. Moreover, it might be relevant to mention 
here that when the uniform control of the derivatives is replaced by
$L^p$-control, an additional dichotomy shows up in the low exponent regime
$0<p<1$ \cite{HedWenn}.

The concept of analytic quasianalyticity is analogous. If we have a function
which is $C^\infty$-smooth on the real line, and also possesses a 
(necessarily unique) holomorphic extension to the upper half-plane, we may 
still ask when the vanishing of all the Taylor coefficients at a given real point 
forces the function to vanish everywhere. The natural condition on the (locally)
uniform control on the growth of the higher order derivatives remains in force,
and the question is asked in terms of the weight sequence controlling the 
growth. It is to be expected that the one-sided holomorphic extension property 
means that less stringent growth control is needed to force the function to vanish 
everywhere.
 And when we are in the opposite world of analytic non-quasianalytic classes, that
 means that that there exist nontrivial functions in the class 
 that are flat at a given real point and holomorphic in the  upper half-plane.  What
 should such functions look like? We cannot as in the classical non-quasianalytic 
case just take such a function and replace it by the $0$ function on one side of
the given point, as that would ruin the one-sided holomorphic extension property.
 Typically, if such flat functions exist, there also exist such flat functions with 
 a decay at the given flat point which from the positive imaginary direction is
 like that of an atomic inner function in the inner-outer factorization theory of the 
 upper half-plane. In \cite{Beu}, Beurling studied weighted spaces on semigroups
 where analytic non-quasianalyticity is a natural dividing property after application 
 of the Fourier-Laplace transform.  Dually, we have instead a space of functions
 that are defined and holomorphic in the upper half-plane but need not have 
 well-defined boundary values on the real line pointwise. 
 Instead, the functions meet a growth condition in the upper half-plane. Since 
 basically Fourier analysis on the half-line amounts to expansion of the functions 
 in terms of the singular atomic inner functions at the point at infinity, it becomes
 natural to ask when this carries over to the setting of a finite point in a stable 
 fashion. This stability would require the various atomic inner functions at the given
 point $x_0\in\R$, naturally indexed by a parameter $\alpha\in\R_{\ge0}$, to 
 generate invariant subspaces $I_\alpha(x_0)$ which get smaller as $\alpha$ 
 increases, and the containments should be strict. 
 This issue appeared later in the work of Korenblum \cite{Kor3} (also  in collaboration 
 with Hayman \cite{HayKor}). It was motivated by Korenblum's earlier work 
 \cite{Kor1}, \cite{Kor2} which extended as far as possible the classical Nevanlinna 
 theory as well as the Beurling invariant subspace theorem to the soft topology slow 
 growth spaces $\mathcal{A}^{-\infty}$ consisting 
 of entire functions of finite exponential type in the context of the Poincar\'e model 
 of the hyperbolic plane. What we will need here is basically the property that the limit 
 of the invariant subspaces $I_\alpha(x_0)$ as $\alpha\to+\infty$ is the trivial 
 subspace $\{0\}$, interpreted in the sense of intersection. 
  
\subsection{A boundary decay uniqueness theorem}
The following is a prototypical radial boundary decay uniqueness theorem.

\begin{thm}
Suppose $F$ is a holomorphic function in the open unit disk $\D$, and that $F$ is
bounded. Suppose, moreover, that we have the radial limit
\[
\lim_{r\to 1^-}(1-r)\log|F(r)|=-\infty.
\] 
Then $F$ must vanish identically in $\D$.
\label{thm:NT-1}
\end{thm}  

This theorem is an immediate consequence of the classical Nevanlinna factorization 
theory, and we refrain from supplying a proof. We could mention, however, that for 
Blaschke products 
\[
B(z)=z^N \prod_j\frac{|a_j|}{a_j}\frac{a_j-z}{1-\bar a_jz},
\]
where $N\in\Z_{\ge0}$ and the sequence $\{a_j\}_j$ from $\D$ meets the Blaschke 
condition
\[
\sum_j (1-|z_j|)<+\infty,
\]
it is not hard to establish that
\[
\limsup_{r\to 1^-}(1-r)\log|B(r)|=0,
\]
so that the only source of a nontrivial limit in this expression could only be the nonzero
factors (outer and singular inner). As it happens, only the singular (in fact atomic) inner
part can contribute. 
In fact,  the examples of atomic inner functions
\[
F(z)=\exp\bigg(-\alpha\frac{1+z}{1-z}\bigg)
\]
for $\alpha\ge0$ show that any finite limit
\[
\lim_{r\to1^-}(1-r)\log|F(r)|=-2\alpha
\]
can be achieved for a nontrivial bounded analytic function. Note that we can use
a conformal transformation (or, more generally, a holomorphic transformation), to
achieve a weakening of the assumptions of the theorem. This idea is one of the basic 
ingredients which helped Korenblum to develop his non-Nevanlinna theory in
\cite{Kor1}, \cite{Kor2}. Here, the more precise setting is in part taken from the paper of 
Hayman and Korenblum \cite{HayKor0}, \cite{HayKor}. 

\begin{thm}
Suppose $F$ is holomorphic in $\D$, and that there exists a nonconstant
holomorphic map $\phi:\D\to\D$ such that $F\circ\phi$ has bounded
Nevanlinna characteristic. Assume that $\phi(r)$ is real for $r$ sufficiently
close to $1$, that $\phi(r)\to1$, and that
\[
0<1-\phi(r)\le C_0(1-r),\qquad r_0\le r<1.
\]
If
\[
\lim_{r\to1^-}(1-r)\log|F(r)|=-\infty,
\]
then $F$ vanishes identically.
\label{thm:NT-2}
\end{thm}

\begin{proof}
The boundary decay and the estimate for $1-\phi(r)$ imply
\[
\lim_{r\to1^-}(1-r)\log|F(\phi(r))|=-\infty.
\]
Write $F\circ\phi=A/B$, where $A,B$ are bounded holomorphic functions
on $\D$ and $B\not\equiv0$. Since $B$ is bounded, the same decay holds
with $A$ in place of $F\circ\phi$. Theorem \ref{thm:NT-1} gives
$A\equiv0$. Hence $F\circ\phi\equiv0$, and the open mapping theorem
implies $F\equiv0$.
\end{proof}

\subsection{Application to the setting of the right half-plane and the proof of
Theorem \ref{thm:main-2}}
\label{ss:Korenblum-1}
We need to apply Theorem \ref{thm:NT-2} in the setting of the right-half plane
rather than the unit disk. In principle, we can use the standard Cayley mapping 
to obtain the required statement from Theorem \ref{thm:NT-2}.

\begin{cor}
Suppose $F$ is holomorphic in the right half-plane $\C_{\Re>0}$ and
there is a nonconstant holomorphic map $\phi:\D\to\C_{\Re>0}$
such that $F\circ\phi$ has bounded Nevanlinna characteristic.
Assume that $\phi(r)$ is positive real for $r$ sufficiently close to $1$,
that $\phi(r)\to0$, and that
\[
0<\phi(r)\le C_0(1-r),\qquad r_0\le r<1.
\]
If
\[
\lim_{x\to0^+}x\log|F(x)|=-\infty,
\]
then $F$ vanishes identically in $\C_{\Re>0}$.
\label{cor:NT-3}
\end{cor}

\begin{proof}
The decay assumption and the estimate on $\phi(r)$ give
$\lim_{r\to1^-}(1-r)\log|F(\phi(r))|=-\infty$.
Represent $F\circ\phi$ as a quotient of bounded holomorphic functions
and apply Theorem \ref{thm:NT-1} to its numerator, as in the proof of
Theorem \ref{thm:NT-2}. The open mapping theorem then gives $F\equiv0$.
\end{proof}

\begin{proof}[Proof of Theorem \ref{thm:main-2}]
As in the proof of Theorem \ref{thm:main-1-0}, we may assume that 
$u=u_{\rm{horz}}=\Rop[f,0]$ is a unilateral horizontal wave solution,
with $f(x)=u(x,0)$ for $x\in\R_{\ge0}$, so that $f$ is continuous and
$f(0)=0$. The given growth control says that
\[
|u(x,y)|\le C(y)\,\exp(x^q),\qquad (x,y)\in\R_{\ge0}\times\calY,
\]
where $C(y)>0$ may depend arbitrarily on $y\in\calY$. Since $0<q<\frac12$
is assumed,
the Laplace transform of $u(\cdot,y)$ is holomorphic in the right half-plane
$\C_{\Re>0}$ for $y\in\calY$, and, in view of Lemma \ref{lem:basic-est-1},
we have the estimate, for $\sigma=\re\zeta>0$ and $y\in\calY$:
\begin{multline}
|\Lop[u(\cdot,y)](\zeta)|\le \int_0^{+\infty}|u(x,y)|\,\e^{-x\re\zeta}\diff x
\le C(y)\int_0^{+\infty}\exp(x^q-\sigma x)\diff x
\\
\le C(y)\,\frac{1}{\sigma}\big(A(q,1)\sigma^{-\frac{q}{2(1-q)}}+B_q\big)
\exp\big(D(q,1)\,\sigma^{-\frac{q}{1-q}}\big),
\label{eq:LG-1}
\end{multline}
where the constants are as in the lemma. Next, in view of \eqref{eq:Ropf0-2},
we know that
\[
\Lop[u(\cdot,y)](\zeta)=\e^{-y/\zeta}\Lop[f](\zeta),
\]
so that \eqref{eq:LG-1} entails that
\begin{equation}
|\Lop[f](\zeta)|\le 
C(y)\,\frac{1}{\sigma}\big(A(q,1)\sigma^{-\frac{q}{2(1-q)}}+B_q\big)
\exp\big(D(q,1)\,\sigma^{-\frac{q}{1-q}}+y\re(1/\zeta)\big),
\qquad y\in\calY.
\label{eq:LG-2}
\end{equation}
It now follows that for $\zeta=\sigma>0$, we have that
\begin{multline}
\sigma\log|\Lop[f](\sigma)|\le 
\sigma\log\frac{C(y)}{\sigma}
\\
+\sigma\log
\big(A(q,1)\sigma^{-\frac{q}{2(1-q)}}+B_q\big)
+D(q,1)\,\sigma^{\frac{1-2q}{1-q}}+y,\qquad
y\in\calY.
\label{eq:LG-2.1}
\end{multline}
We let $\sigma\to 0^+$, and observe that since $0<q<\frac12$, 
it follows from \eqref{eq:LG-2.1} that
\begin{equation}
\limsup_{\sigma\to0^+}\sigma\log|\Lop[f](\sigma)|\le y,\qquad
y\in\calY,
\label{eq:LG-3}
\end{equation}
and since $y\in\calY$ is arbitrary, and $\inf\calY=-\infty$ holds, it
then follows that 
\begin{equation}
\lim_{\sigma\to0^+}\sigma\log|\Lop[f](\sigma)|=-\infty.
\label{eq:LG-4}
\end{equation}
To apply Hayman--Korenblum's Theorem A and Definition 2
\cite{HayKor}, set $F=\Lop[f]$ and transfer the right half-plane to the
disk by
\[
H(z)=F\left(\frac{1-z}{1+z}\right),\qquad z\in\D.
\]
For $\zeta=(1-z)/(1+z)$ we have
\[
\re\zeta=\frac{1-|z|^2}{|1+z|^2}\ge\frac{1-|z|}{2}.
\]
Taking any fixed $y\in\calY$ in \eqref{eq:LG-2}, and discarding the
nonpositive term $y\,\re(1/\zeta)$, yields, with $a=q/(1-q)$,
\[
\log^+|H(z)|\le C+C'(1-|z|)^{-a}
+C'\log\frac{\mathrm e}{1-|z|},\qquad z\in\D.
\]
Here $C,C'$ are constants depending on $q$ and the chosen $y$.
After enlarging the right-hand side to a positive continuous increasing
radial majorant $k(r)$, we obtain
\[
\int_0^1\sqrt{\frac{k(r)}{1-r}}\diff r<+\infty,
\qquad 0<q<\tfrac12.
\]
Theorem A and Definition 2 of Hayman--Korenblum therefore give a
holomorphic map $\psi:\D\to\D$, real and increasing on $[0,1]$, with
$\psi(r)\to1$ and
\[
1-\psi(r)\le C_1(1-r),
\]
such that $H\circ\psi$ has bounded Nevanlinna characteristic.
The uniform characteristic estimate is clarified in equations (5)--(6)
of \cite{HayKor}; the conformal-mapping construction underlying the
auxiliary map is given in Section 5 of \cite{HayKor0}. In particular,
that construction allows $\psi$ to be chosen conformal onto its image.
Put
\[
\phi(w)=\frac{1-\psi(w)}{1+\psi(w)}.
\]
Then $F\circ\phi=H\circ\psi$ has bounded characteristic, and
$\phi(r)>0$ tends to zero with $\phi(r)\le C_1(1-r)$.
Thus all the hypotheses of Corollary \ref{cor:NT-3} hold in view of
\eqref{eq:LG-4}.
By the work of Hayman and Korenblum, then, we may conclude in the
setting of Corollary \ref{cor:NT-3} that $F=\Lop[f]$
must vanish identically, and hence $f=0$ almost everywhere on $\R_{\ge0}$,
so that $u=\Rop[f,0]=\Rop[0,0]=0$ in the quarter-plane 
$\R_{\ge0}\times\R_{\le0}$, as claimed.
\end{proof}

\begin{rem}
The Hayman--Korenblum integral condition used above is a version of the
criterion for analytic non-quasianalyticity.
\end{rem}

\section{Uniqueness in the skewed exponential growth case
 $\frac12<q<1$}
\label{sec:skewed-1}
The relationship between upper and lower radial growth bounds for harmonic
functions has a distinguished history. An early contribution is Mary
Cartwright's 1933 paper \cite{Cartwright1933}, see also Borichev's later work
\cite{Borichev2003}.
The connection with holomorphic functions comes from the fact that
$\log|F|$ is real-valued and harmonic wherever the holomorphic function
$F$ is zero-free. Upper and lower bounds for a harmonic function may
therefore be interpreted as bounds on the growth and decay of such an $F$.
In our setting, however, the holomorphic functions need not be zero-free,
so $\log|F|$ is generally only subharmonic. The harmonic-function results
of Cartwright and Borichev consequently do not apply directly. The proof
of our uniqueness theorem in the regime $\frac12<q<1$ does not require
that the holomorphic function be zero-free.

\subsection{Skewed exponential growth}
As our starting point, we have a solution $u$ to the Klein-Gordon equation 
$u''_{xy}+u=0$ in the spacelike quarter-plane $\R_{>0}\times\R_{<0}$ which
extends continuously to the closed quarter-plane $\R_{\ge0}\times\R_{\le0}$,
with vanishing vertical boundary values:
\[
u(0,y)=0,\qquad y\in\R_{\le0}.
\]
As explained in the proof of Theorem \ref{thm:main-1-0}, $u=\Rop[f,0]$ if the
continuous function $f$ is given by $f(x)=u(x,0)$. 
We assume that $\frac12<q<1$, and that
\begin{equation}
|u(x,y)|\le C_0 \exp(\theta_1 |x|^q+\theta_2|y|^{q''}),\qquad 
(x,y)\in\R_{\ge0}\times\calY,
\label{eq:skewed-1}
\end{equation}
holds for some positive constants $C_0$, and $\theta_1,\theta_2$, where
$q''=q/(2q-1)$. Here, $\calY$ is a subset of the half-line $\R_{\le0}$ 
with the properties prescribed by Theorem \ref{thm:main-1-2}.
We apply the Laplace transform with respect to the $x$ variable:
\[
\Lop[u(\cdot,y)](\zeta)=\int_0^{+\infty}u(x,y)\,\e^{-x\zeta}\diff x,\qquad
\Re\, \zeta>0,
\]
which defines a holomorphic function in the indicated half-plane for each 
$y\in\calY$. In view of \eqref{eq:Ropf0-2},
\[
\Lop[u(\cdot,y)](\zeta)=\e^{-y/\zeta}\Lop[f](\zeta),
\]
and by the estimate \eqref{eq:skewed-1}, we have, for $\sigma=\re\zeta>0$,
\begin{multline}
|\e^{-y/\zeta}\Lop[f](\zeta)|=\Lop[u(\cdot,y)](\zeta)|\le 
\int_0^{+\infty}|u(x,y)|\,\e^{-\sigma x}\diff x
\\
\le C_0\exp(\theta_2|y|^{q''})\int_0^{+\infty}\exp(\theta_1 x^q-\sigma x)\,\diff x,
\qquad y\in\calY,
\end{multline}
so that 
\begin{equation}
|\Lop[f](\zeta)|
\le C_0\exp\big(\theta_2|y|^{q''}+\re(y/\zeta)\big)
\int_0^{+\infty}\exp(\theta_1 x^q-\sigma x)\,\diff x,\qquad y\in\calY,
\end{equation}
where again we write $\sigma=\re\zeta>0$. Here, we may optimize over 
$y\in\calY$, which yields
\begin{equation}
|\Lop[f](\zeta)|
\le C_0\exp\Big(\inf_{y\in\calY}\big\{\theta_2|y|^{q''}+\re(y/\zeta)\big\}\Big)
\int_0^{+\infty}\exp(\theta_1 x^q-\sigma x)\,\diff x.
\label{eq:optimum-1}
\end{equation}
 If we write
\[
\psi(y)=\theta_2|y|^{q''}-|y|\,\re(1/\zeta),
\]
we may first minimize over all $y\in\R_{\le0}$ and then hope that the 
minimum over $y\in\calY$ is not too much bigger. 
The minimum over $y\in\R_{\le0}$ is attained at the point 
\begin{equation}
y_\star(\zeta,\theta_2)=-\bigg(\frac{\re(1/\zeta)}{\theta_2 q''}\bigg)^{\frac{1}{q''-1}},
\label{eq:ystar-min-1}
\end{equation}
where we suppress indication of the variable $q$ as it is kept fixed, and then
\[
\psi(y_\star(\zeta,\theta_2))=-E(q,\theta_2)(\re(1/\zeta))^{\frac{q}{1-q}}, 
\]
where
\[
E(q,\theta_2):=\bigg(\frac1{q}-1\bigg)(q''\theta_2)^{1-\frac{q}{1-q}}. 
\]
At some other point $y\in\R_{\le0}$, we may apply a version of 
Taylor's formula:
\begin{multline}
\psi(y)=\psi(y_\star(\zeta,\theta_2))+(y-y_\star(\zeta,\theta_2))
\psi'(y_\star(\zeta,\theta_2))+\frac12(y-y_\star(\zeta,\theta_2))^2\psi''(\xi)
\\
=\psi(y_\star(\zeta,\theta_2))+\frac12(y-y_\star(\zeta,\theta_2))^2\psi''(\xi),
\end{multline}
since $\psi'(y_\star(\zeta,\theta_2))=0$ must hold at the minimum. Here, 
$\xi$ stands for an intermediate point between $y_\star(\zeta,\theta_2)$
and $y$. Since 
\[
\psi''(y)=\theta_2 q''(q''-1)|y|^{q''-2},
\]
it now follows that
\begin{multline}
\psi(y)=\psi(y_\star(\zeta,\theta_2))+\frac12(y-y_\star(\zeta,\theta_2))^2\psi''(\xi)
\\
=-E(q,\theta_2)(\re(1/\zeta))^{\frac{q}{1-q}}+
\frac12(y-y_\star(\zeta,\theta_2))^2\theta_2 q''(q''-1)|\xi|^{q''-2}.
\label{eq:psiexp-1}
\end{multline}

\begin{lem}
Suppose $\calY\subset\R_{\le0}$ and the point $y_1\in\R_{\le0}$ are as in the 
formulation of Theorem \ref{thm:main-1-2}.
If $y_2:=\min\{-1,y_1,-(2M)^{2/q''}\}$, and if $y_\star< (M+1)y_2$,
then there exists a point $y\in\calY$ with $y<y_2$ such that 
\[
y-M|y|^{1-\frac{q''}{2}}\le y_\star\le y+M|y|^{1-\frac{q''}{2}},
\]
and hence 
\[
2y_\star\le y\le \frac{2}{3}y_\star.
\]
\label{lem:ystar-1}
\end{lem}

\begin{proof}
If $y\in\calY$ with $y_2\le y\le0$, we use the inequality $0\le R(y)\le |y|$ to show that
we have the containment of intervals
\[
[y-MR(y),y+MR(y)]\subset[(M+1)y_2, (1-M)y_2],
\]
so that in view of the assumed inequality $y_\star<(M+1)y_2$, it follows that
\[
y_\star\notin [y-MR(y),y+MR(y)],\qquad y\in\calY,\,\,\,y_2\le y\le0,
\]
so that in view of the asymptotic $q$-covering property, there must instead exist a
$y\in\calY$ with $y<y_2$ such that 
\[
y_\star\in [y-MR(y),y+MR(y)],
\]
and since $y<y_2\le-1$, we know that $R(y)=|y|^{1-\frac{q''}{2}}$, and hence
\begin{equation}
y(1+M|y|^{-\frac{q''}{2}})=y-M|y|^{1-\frac{q''}{2}}\le y_\star\le y+M|y|^{1-\frac{q''}{2}}
=y(1-M|y|^{-\frac{q''}{2}}).
\label{eq:interval-1}
\end{equation}
Moreover, since $y<y_2\le -(2M)^{2/q''}$, it follows that $|y|^{-q''/2}<(2M)^{-1}$, 
and hence, by \eqref{eq:interval-1},
\begin{equation}
\frac{3}{2}y\le y_\star\le \frac{1}{2}y,
\label{eq:interval-2}
\end{equation}
which is equivalent to
\begin{equation}
2y_\star\le y\le \frac{2}{3}y_\star,
\label{eq:interval-3}
\end{equation}
as claimed. 
\end{proof}

\begin{lem}
Suppose $y,y_\star\in\R_{<0}$ are such that $y\le-1$ and 
\[
y-M|y|^{1-\frac{q''}{2}}\le y_\star\le y+M|y|^{1-\frac{q''}{2}},
\]
and 
\[
2y_\star\le y\le \frac{2}{3}y_\star.
\]
Then, if $\xi$ is an intermediate point between $y_\star$ and $y$, it
follows that
\[
0\le (y-y_\star)^2|\xi|^{q''-2}\le 2^{|q''-2|} M^2.
\]
\label{lem:ystar-2}
\end{lem}

\begin{proof}
Since
\[
(y-y_\star)^2\le M^2|y|^{2-q''},
\]
and, in addition, we check that
\[
|\xi|^{q''-2}\le 2^{|q''-2|}|y|^{q''-2},
\]
so that, consequently,
\[
0\le(y-y_\star)^2|\xi|^{q''-2}\le 2^{|q''-2|}M^2,
\]
as claimed.
\end{proof}

\begin{proof}[Proof of Theorem \ref{thm:main-1-2}]
By Remark \ref{rem:observation-translation}, we may assume that
$0\in\calY$.
We let $y_\star(\zeta,\theta_2)$ be given by \eqref{eq:ystar-min-1}, and assume
that $y_\star(\zeta,\theta_2)<(M+1)y_2$, where $y_2$ is as in Lemma \ref{lem:ystar-1}.
Then, according to Lemma \ref{lem:ystar-1}, there exists a point $y\in\calY$ with  $y<y_2$
such that in accordance with Lemma \ref{lem:ystar-2} combined with equation 
\eqref{eq:psiexp-1}, we have
\begin{equation}
\psi(y)\le-E(q,\theta_2)(\re(1/\zeta))^{\frac{q}{1-q}}+
2^{|q''-2|-1}\theta_2 q''(q''-1)M^2.
\end{equation}
As we insert this point into the estimate \eqref{eq:optimum-1}, we find that
\begin{equation}
|\Lop[f](\zeta)|
\le C_1\exp\Big(-E(q,\theta_2)(\re(1/\zeta))^{\frac{q}{1-q}}
\Big)
\int_0^{+\infty}\exp(\theta_1 x^q-\sigma x)\,\diff x,\qquad
\zeta\in \Omega_2,
\label{eq:optimum-2}
\end{equation}
where
\[
C_1:=C_0\exp\big(2^{|q''-2|-1}\theta_2 q''(q''-1)M^2\big),
\]
and $\Omega_2$ stands for the domain induced by the requirement that
$y_\star(\zeta,\theta_2)<(M+1)y_2$, that is,
\begin{equation}
\Omega_2:=\Big\{\zeta\in\C_{\Re>0}:\,\,\re(1/\zeta)>\theta_2 q''
(M+1)^{q''-1}|y_2|^{q''-1}\Big\},
\end{equation}
which is a horocyclic disk tangential to the imaginary line at the origin.
We may now apply Lemma \ref{lem:basic-est-1} to estimate the integral in 
\eqref{eq:optimum-2}, which gives, for $\sigma\in\R_{>0}\cap\Omega_2$,
\begin{equation}
|\Lop[f](\sigma)|
\le
\frac{C_1}{\sigma}\big(A(q,\theta_1) \sigma^{-\frac{q}{2(1-q)}}+B(q)\big)
\exp\Big((D(q,\theta_1)-E(q,\theta_2))\,\sigma^{-\frac{q}{1-q}}\Big).
\label{eq:skewed-2}
\end{equation}
while more generally, we have 
\begin{multline}
|\Lop[f](\zeta)|
\le
\frac{C_1}{\sigma}\big(A(q,\theta_1) \sigma^{-\frac{q}{2(1-q)}}+B(q)\big)
\\
\times\exp\Big((D(q,\theta_1))\,\sigma^{-\frac{q}{1-q}}-E(q,\theta_2)
(\re(1/\zeta))^{\frac{q}{1-q}}\Big),
\label{eq:skewed-3}
\end{multline}
for $\zeta\in\Omega_2$. Since
\[
\Re\frac{1}{\zeta}=\frac{\re\zeta}{|\zeta|^2}=\frac{\sigma}{|\zeta|^2}
=\frac{1}{\sigma}\,\bigg(\frac{\sigma}{|\zeta|}\bigg)^2, 
\]
it follows from \eqref{eq:skewed-3} that
\begin{equation}
|\Lop[f](\zeta)|
\le
\frac{C_1}{\sigma}\big(A(q,\theta_1) \sigma^{-\frac{q}{2(1-q)}}+B(q)\big),
\qquad \zeta\in\Omega_2\cap\Omega_3,
\label{eq:skewed-4}
\end{equation}
 holds, where $\Omega_3$ denotes the domain in the 
 half-plane $\C_{\Re>0}$ where
 \begin{equation}
 D(q,\theta_1)< E(q,\theta_2)\,
 \bigg(\frac{\sigma}{|\zeta|}\bigg)^{\frac{2q}{1-q}},
 \end{equation}
 that is, where 
 \begin{equation}
 \frac{\sigma^2}{|\zeta|^2}> 
 \bigg(\frac{D(q,\theta_1)}{E(q,\theta_2)}\bigg)^{\frac{1}{q}-1}=
 {(q\theta_1)^{\frac1q}}{(q''\theta_2)^{\frac{1}{q''}}},
 \end{equation}
 which is an angular opening with apex at the origin. By assumption, we 
 have that
\[ 
{(q\theta_1)^{\frac1q}}{(q''\theta_2)^{\frac{1}{q''}}}<\sin^2\frac{\pi}{2q}=
\cos^2\frac{\pi(1-q)}{2q}<1,
\]
so that $D(q,\theta_1)<E(q,\theta_2)$ and hence we have exponential decay
at the given rate along the real line as $\sigma\to0^+$ in \eqref{eq:skewed-2}. 
If we write
\[
{(q\theta_1)^{\frac1q}}{(q''\theta_2)^{\frac{1}{q''}}}=\cos^2\frac{\pi\alpha}{2},
\quad\text{where}\quad \frac{1-q}{q}<\alpha<1,
\]
then $\alpha=\alpha(q,\theta_1,\theta_2)$, and the region 
$\Omega_3$ is an angular opening toward
the right, centered at the origin, symmetric about the real line, with 
aperture $\pi\alpha$. Under the change-of-variables $\zeta=\eta^\alpha$, then
$\zeta\in \Omega_3$ if and only if $\eta\in\C_{\Re>0}$. Here,
the power is defined based on the principal branch of the logarithm.
We put
\[
F(\eta):=\Lop[f](\eta^\alpha),
\]
which is then holomorphic in the symmetric angular opening of aperture 
$\pi/\alpha$, and hence in particular in the right half-plane $\C_{\Re>0}$.
It follows from the estimate \eqref{eq:skewed-4} that uniformly,
\begin{equation}
|F(\eta)|
=
\Ordo\big(|\eta|^{-{\alpha}}\big(|\eta|^{-\frac{\alpha q}{2(1-q)}}+1)\big),
\qquad \eta\in\C_{\Re>0}\cap\tilde\Omega_2, 
\label{eq:skewed-5}
\end{equation}
provided that $\eta\in\tilde\Omega_2$ means that 
$\zeta=\eta^\alpha\in\Omega_2$.
At the same time, we have exponential decay in accordance with 
\eqref{eq:skewed-2}:
\begin{equation}
|F(\rho)|
=
\Ordo\big(\rho^{-\alpha}\big(\rho^{-\frac{\alpha q}{2(1-q)}}+1\big)
\exp\Big((D(q,\theta_1)-E(q,\theta_2))\,\rho^{-\frac{\alpha q}{1-q}}\Big)\Big),
\label{eq:skewed-6}
\end{equation}
uniformly over $\rho\in\R_{>0}$. Since $D(q,\theta_1)<E(q,\theta_2)$ and
$\alpha q/(1-q)>1$, it is an immediate consequence of \eqref{eq:skewed-6}
that
\begin{equation}
\lim_{\rho\to0^+}\rho\log|F(\rho)|=-\infty.
\label{eq:skewed-7}
\end{equation}
The estimate \eqref{eq:skewed-5} allows polynomial growth near the
boundary point $0$. To remove this growth, choose an integer
\[
N>\alpha\Big(1+\frac{q}{2(1-q)}\Big)
\]
and put
\[
G(\eta):=\Big(\frac{\eta}{1+\eta}\Big)^N.
\]
The function $G$ is bounded and holomorphic in $\C_{\Re>0}$ and has no
zeros there. In view of \eqref{eq:skewed-5}, the product $FG$ is bounded
in $\C_{\Re>0}\cap\tilde\Omega_2$: the factor $G$ absorbs the
polynomial growth near $0$, and away from $0$ the estimate is uniform.
Moreover, $\rho\log|G(\rho)|\to0$ as $\rho\to0^+$, so
\eqref{eq:skewed-7} implies
\[
\lim_{\rho\to0^+}\rho\log|(FG)(\rho)|=-\infty.
\]
The simply connected domain $\C_{\Re>0}\cap\tilde\Omega_2$ agrees
with the right half-plane locally near $0$. Consequently, a conformal
map $\phi:\D\to\C_{\Re>0}\cap\tilde\Omega_2$ with $\phi(1)=0$
extends conformally across $1$. The conditions of Theorem
\ref{thm:NT-1} apply to $FG$, and hence $FG=0$. Since $G$ has no
zeros in the domain, $F=0$ there and therefore throughout
$\C_{\Re>0}$ by the identity theorem. It follows that
$\Lop[f]=0$ and hence $f=0$.
This completes the proof of the theorem.
\end{proof}

\subsection{A trigonometric minimax inequality}
We next establish a trigonometric minimax inequality that will be useful
in the construction of counterexamples in the following subsection.

\begin{prop}
Let $p>1$. Then for all $0\le t\le \tfrac12\pi$,
\begin{equation}
\tag{$a$}
\label{eq:counter-skew-main}
\cos^{2p}t-2\bigl(\cos(\pi/(2p))\bigr)^{p-1}
\sin\!\bigl(\pi/(2p)\bigr)\cos^pt\cos(pt)\le \bigl(\cos(\pi/(2p))\bigr)^{2p}.
\end{equation}
Equality holds at $t=\pi/(2p)$. Consequently, for every $E>0$,
\begin{equation}
\tag{$b$}
\label{eq:counter-skew-minimax}
\inf_{a>0}\max_{|t|\le\tfrac12\pi}
\bigl\{E\cos^{2p}t-a\cos^pt\cos(pt)\bigr\}
=E\bigl(\cos(\pi/(2p))\bigr)^{2p},
\end{equation}
with the infimum attained at
\[
a_*=2E\bigl(\cos(\pi/(2p))\bigr)^{p-1}
\sin\!\bigl(\pi/(2p)\bigr).
\]
\end{prop}

\begin{proof}
The inequality has a geometric interpretation: the curve
$z(t)=\cos^pt\,\e^{\imag pt}$ lies in the disk with center
$\bigl(\cos(\pi/(2p))\bigr)^{p-1}\sin(\pi/(2p))$ and radius
$\bigl(\cos(\pi/(2p))\bigr)^{p-1}$.
Indeed,
\[
z'(t)=\imag p\cos^{p-1}t\,\e^{\imag(p+1)t}.
\]
Thus the tangent rotates uniformly while its magnitude decreases.
We use a Wronskian to control the first oscillation sharply, and
radial decay to control all later ones.

Write
\[
G(t)=\cos^{2p}t-2\bigl(\cos(\pi/(2p))\bigr)^{p-1}
\sin(\pi/(2p))\cos^pt\cos(pt).
\]
By evenness, consider $0\le t\le\tfrac12\pi$. We have
$G\bigl(\pi/(2p)\bigr)=\bigl(\cos(\pi/(2p))\bigr)^{2p}$ and
\begin{equation}\label{eq:counter-skew-derivative}
G'(t)=2p\cos^{p-1}t\Bigl(\bigl(\cos(\pi/(2p))\bigr)^{p-1}
\sin(\pi/(2p))\sin((p+1)t)-\cos^pt\sin t\Bigr).
\end{equation}
On $0<t<\pi/(p+1)$, the quotient
\[
H(t)=\frac{\cos^pt\sin t}{\sin((p+1)t)}
\]
is strictly increasing. To see this, put $n=p+1$,
$h(t)=\sin t\cos^{n-1}t$, $v(t)=\sin(nt)$, and
$W=hv'-h'v$. Then
\[
h''+n^2h=(n-1)(n-2)\frac{h}{\cos^2t}>0,
\qquad W'=-(h''+n^2h)v<0,
\]
while $W(0)=0$. Hence $(v/h)'=W/h^2<0$ and $H=h/v$
increases. At the distinguished angle,
\[
H\bigl(\pi/(2p)\bigr)=\bigl(\cos(\pi/(2p))\bigr)^{p-1}\sin(\pi/(2p)).
\]
The derivative formula shows that $G$ increases
up to $\pi/(2p)$ and decreases until $\pi/(p+1)$.
It continues to decrease on
$\pi/(p+1)\le t\le\min\{2\pi/(p+1),\tfrac12\pi\}$,
where $\sin((p+1)t)\le0$. This proves the assertion for $1<p\le3$.

For $p>3$ and $t\ge2\pi/(p+1)>\pi/p$, radial decay gives
\[
\cos^pt\le\cos^p(\pi/p)
\le\bigl(\cos(\pi/(2p))\bigr)^{4p},
\]
where the last inequality follows from
$\cos(\pi/p)=2\cos^2(\pi/(2p))-1\le\cos^4(\pi/(2p))$.
Bernoulli's inequality now gives
\[
\begin{aligned}
\bigl(\cos(\pi/(2p))\bigr)^{-(3p-1)}
&\ge1+\frac{3p-1}{2}\sin^2(\pi/(2p))\\
&\ge1+\sin(\pi/(2p)),\qquad p>3,
\end{aligned}
\]
since $\sin(\pi/(2p))\ge1/p\ge2/(3p-1)$.
Consequently,
\[
\cos^pt\le\bigl(\cos(\pi/(2p))\bigr)^{p-1}
\Bigl(1-\sin(\pi/(2p))\Bigr).
\]
By the triangle inequality, the curve therefore stays in the stated
disk for all remaining $t$. Squaring the disk inequality proves
\eqref{eq:counter-skew-main} on the entire interval.

Finally, evaluating the minimax expression at $t=\pi/(2p)$
gives the lower bound $E\bigl(\cos(\pi/(2p))\bigr)^{2p}$
for every $a>0$. Taking $a=a_*$ and using
\eqref{eq:counter-skew-main} gives the matching upper bound,
and hence property \eqref{eq:counter-skew-minimax} follows.
\end{proof}

\subsection{Counterexamples for the skewed exponential growth
condition}\label{sec:counter-skew}

We now turn to the complementary question of constructing nontrivial
solutions when the parameter condition for uniqueness is no longer satisfied.
The trigonometric minimax inequality of Proposition~\ref{eq:counter-skew-main}
reflects the same angular geometry that governs the uniqueness argument
in Subsection~6.1, and leads to the same sharp threshold.

\begin{proof}[Proof of Theorem \ref{thm:intermediate-sharpness}]
Set
\[
p=\frac{q}{1-q}>1,\qquad q=\frac{p}{p+1},\qquad
q''=\frac{p}{p-1},
\]
and define
\begin{equation}\label{eq:counter-skew-DE}
D=\frac{p^p}{(p+1)^{p+1}}\theta_1^{p+1},\qquad
E=\frac{(p-1)^{p-1}}{p^p\theta_2^{p-1}}.
\end{equation}
Since $p=q/(1-q)$, we have
\[
\bigl(\cos(\pi/(2p))\bigr)^{2p}
=\bigl(\sin(\pi/(2q))\bigr)^{2p}.
\]
Thus the parameter condition in Theorem \ref{thm:intermediate-sharpness} is equivalent to
\begin{equation}\label{eq:counter-skew-gap}
D\ge E\cos^{2p}\!\left(\frac{\pi}{2p}\right).
\end{equation}
Choose
\[
a=2E\cos^{p-1}\!\left(\frac{\pi}{2p}\right)\sin\!\left(\frac{\pi}{2p}\right)>0.
\]
On the right half-plane take the holomorphic principal branch of
$\zeta^{-p}$ and set
\[
F_a(\zeta)=\frac{\exp(-a\zeta^{-p})}{(1+\zeta)^2}.
\]
For $x\ge0$ and $y\le0$, define $u(x,y)$ by the absolutely
convergent Bromwich integral
\begin{equation}\label{eq:counter-skew-bromwich}
u(x,y)=\frac1{2\pi\imag}\int_{\sigma-\imag\infty}^{\sigma+\imag\infty}
\frac{\exp(x\zeta-y/\zeta-a\zeta^{-p})}{(1+\zeta)^2}
\,\diff\zeta,\qquad \sigma>0.
\end{equation}
The integrand is holomorphic for $\re\zeta>0$,
and decays as $O(|\im\zeta|^{-2})$ on fixed vertical
lines, uniformly for $(x,y)$ in compact sets. Standard contour
shifting shows that \eqref{eq:counter-skew-bromwich} is independent of
$\sigma>0$.

We next prove the growth bound. Write $\zeta=\sigma(1+\imag\tan t)$,
$-\pi/2<t<\pi/2$. Then
\[
\re\frac1\zeta=\sigma^{-1}\cos^2t,
\qquad
\re\zeta^{-p}=\sigma^{-p}\cos^pt\cos(pt).
\]
Young's inequality, with conjugate exponents $q''=p/(p-1)$ and $p$,
gives for $|y|\ge0$ and $r\ge0$,
\begin{equation}\label{eq:counter-skew-young}
|y|r\le\theta_2|y|^{q''}+Er^p.
\end{equation}
By the preceding trigonometric proposition and the choice of $a$,
\begin{align*}
|y|\re\frac1\zeta-a\re\zeta^{-p}
&\le \theta_2|y|^{q''}
+E\left(\re\frac1\zeta\right)^p
-a\re\zeta^{-p}\\
&\le\theta_2|y|^{q''}+D\sigma^{-p},
\end{align*}
where we used $E\cos^{2p}\!\left(\frac{\pi}{2p}\right)\le D$.
Taking absolute values in
\eqref{eq:counter-skew-bromwich} yields
\begin{align}
|u(x,y)|
&\le\frac{\e^{\sigma x+\theta_2|y|^{q''}+D\sigma^{-p}}}{2\pi}
\int_{-\infty}^{+\infty}\frac{dt}{(1+\sigma)^2+t^2}\notag\\
&=\frac{\e^{\sigma x+\theta_2|y|^{q''}+D\sigma^{-p}}}
{2(1+\sigma)}
\le\frac12\e^{\sigma x+\theta_2|y|^{q''}+D\sigma^{-p}}.
\label{eq:counter-skew-pointwise}
\end{align}
For $x>0$,
\[
\inf_{\sigma>0}(\sigma x+D\sigma^{-p})
=(p+1)\left(\frac{D}{p^p}\right)^{1/(p+1)}x^{p/(p+1)}.
\]
By \eqref{eq:counter-skew-gap} and \eqref{eq:counter-skew-DE}, the
coefficient of $x^q$
is exactly $\theta_1$. Hence the growth estimate in Theorem
\ref{thm:intermediate-sharpness}
follows (even with $C=1/2$). At $x=0$, letting
$\sigma\to+\infty$ in \eqref{eq:counter-skew-pointwise} gives $u(0,y)=0$.

For completeness, we justify the differential equation and
regularity without differentiating the Bromwich integral
under the integral sign. Expand
\[
\exp(-y/\zeta-a\zeta^{-p})
=\sum_{j,k\ge0}\frac{(-y)^j(-a)^k}{j!k!}\zeta^{-j-pk}.
\]
Since $\Lop^{-1}[(1+\zeta)^{-2}](x)=x\,\e^{-x}$, the
inverse transform is the locally uniformly convergent series
\begin{align}\label{eq:counter-skew-series}
u(x,y)=x\,\e^{-x}+\sum_{\substack{j,k\ge0\\j+k>0}}
\frac{(-y)^j(-a)^k}{j!k!\,\Gamma(j+pk)}
\int_0^x(x-t)\,\e^{-(x-t)}t^{j+pk-1}\,dt.
\end{align}
Indeed, for $\alpha=j+pk>0$ the convolution integral divided by
$\Gamma(\alpha)$ is bounded by
$ x^{\alpha+1}/\Gamma(\alpha+2)$ on $0\le x\le X$.
Together with the factors $1/(j!k!)$, this proves absolute local
uniform convergence. Differentiation of the convolution kernels on
compact subsets of $x>0$, combined with the same Gamma-function
estimates, justifies the derivatives used below. In particular,
$u$ is smooth for $x>0$, and continuous up to $x=0$.
Taking the Laplace transform in $x$ gives
\[
\Lop[u(\cdot,y)](\zeta)=\e^{-y/\zeta}F_a(\zeta).
\]
Consequently, we have that
\[
\Lop[u'_y(\cdot,y)](\zeta)=-\zeta^{-1}\Lop[u(\cdot,y)](\zeta),
\quad
\Lop[u''_{xy}(\cdot,y)](\zeta)
=\zeta\Lop[u'_y(\cdot,y)](\zeta)-u'_y(0,y)
=-\Lop[u(\cdot,y)](\zeta),
\]
where $u'_y(0,y)=0$ follows from \eqref{eq:counter-skew-series}.
Uniqueness of Laplace transforms implies $u''_{xy}+u=0$.
Finally $u(\cdot,0)$ is not identically zero because its Laplace
transform $F_a$ is nonzero. This proves the proposition.
\end{proof}

The argument establishes nonuniqueness for the entire half-line
$y\le0$ under the parameter condition of Theorem \ref{thm:intermediate-sharpness},
including equality. The minimizing coefficient in the trigonometric
estimate is attained, so no limiting argument or parameter slack is needed.

\section{Uniqueness in the superlinear skewed-growth case $q>1$}
\label{sec:superlinear-uniqueness}

\subsection{The generalized Laplace transform in the exponentially superlinear regime $q>1$}

For superexponentially growing initial data, the ordinary Laplace transform in
$x$ need not converge. Instead, we Laplace-transform the solution in the
opposite characteristic variable, where the growth exponent $q''<1$.
This produces an exact resolvent identity involving a moving truncation
of the Laplace transform of the initial datum.
Related methods involving almost-analytic extensions and Laplace transforms
appear in the work of Dynkin \cite{Dyn93} and, e.g., in the paper of
Borichev and Hedenmalm \cite{BorichevHedenmalm1995}.
These references provide methodological context; the argument here
instead relies on an exact Laplace-transform identity and does not
require a generalized holomorphic transform.
Throughout this subsection we retain the original coordinates and the
Klein-Gordon equation $u''_{xy}+u=0$, rather than reflecting the
second coordinate. We use $|y|$ in the Volterra kernel and integrate
over $t\in\R_{\ge0}$ when taking the Laplace transform of $u(x,-t)$.

\begin{lem}[Volterra representation]
\label{lem:superlinear-volterra}
Suppose $u$ is continuous on $\R_{\ge0}\times\R_{\le0}$, satisfies
$u''_{xy}+u=0$ distributionally in the interior, and $u(0,y)=0$.
With $f(x)=u(x,0)$, we have
\[
 u(x,y)=f(x)+\sum_{n=1}^{+\infty}\frac{|y|^n}{n!(n-1)!}
 \int_0^x(x-t)^{n-1}f(t)\,\diff t.
\]
Equivalently, the sum can be expressed using the modified Bessel function
$I_1$ as
\[
 u(x,y)=f(x)+\int_0^x\sqrt{\frac{|y|}{x-t}}\,
 I_1\!\left(2\sqrt{|y|(x-t)}\right)f(t)\,\diff t.
\]
The series converges locally uniformly, and the integrand at $t=x$
is interpreted by continuity.
\end{lem}
\begin{proof}
Integrating the distributional equation and using continuity gives
\[
 u(x,y)=f(x)+\int_0^x\int_y^0u(t,s)\,\diff s\,\diff t.
\]
Indeed, the difference has zero mixed derivative and vanishes on both
coordinate axes. Iterating the double-integration operator on a compact
rectangle gives a remainder bounded by
$\|u\|_\infty(X|Y|)^N/(N!)^2$, which tends to zero.
The iterated kernels give the asserted series, and the power series
for $I_1$ gives the alternative representation.
\end{proof}

The next two lemmas provide the support principle needed for the uniqueness argument.
We first recall the classical one-sided Laplace support theorem and then
show how it applies to Laplace integrals with a moving truncation, even
when the full Laplace transform need not converge.

\begin{lem}
\label{lem:superlinear-support}
Let $g\in L^1([0,A])$ be nonzero and let
$a=\inf\operatorname{supp}g$. Then
\[
 \limsup_{\sigma\to+\infty}\frac1\sigma
 \log\left|\int_0^A \e^{-\sigma t}g(t)\,\diff t\right|=-a.
\]
\end{lem}
\begin{proof}
This is the classical one-sided Laplace support theorem for compactly
supported distributions (see \cite{DoetschLaplace}), applied to the
compactly supported $L^1$ function $g$. In particular, it remains valid
for complex-valued $g$ and requires no sign condition.
\end{proof}

We next extend this support principle to Laplace integrals with a moving
truncation. This permits us to recover support information even when the
growth of the function prevents convergence of its ordinary Laplace transform.

\begin{lem}
\label{lem:superlinear-moving}
Let $q>1$, $\theta>0$, and $0<\alpha<\theta^{-1/(q-1)}$.
Suppose $f$ is locally integrable on $\R_{\ge0}$ and
$|f(t)|\le C\exp(\theta t^q)$. Put $q'=q/(q-1)$.
If for some $c>0$,
\[
 \int_0^{\alpha\sigma^{q'-1}}\e^{-\sigma t}f(t)\,\diff t
 =\Ordo(\e^{-c\sigma^{q'}}),\qquad\sigma\to+\infty,
\]
then $f=0$ almost everywhere.
\end{lem}
\begin{proof}
Set $\kappa=1-\theta\alpha^{q-1}>0$.
For every fixed $A>0$ and all sufficiently large $\sigma$,
\[
 \left|\int_A^{\alpha\sigma^{q'-1}}\e^{-\sigma t}f(t)\,\diff t\right|
 \le C\int_A^{\alpha\sigma^{q'-1}}\e^{-\kappa\sigma t}\,\diff t
 \le\frac{C}{\kappa\sigma}\e^{-\kappa A\sigma}.
\]
Consequently, with $L_A(\sigma)=\int_0^A \e^{-\sigma t}f(t)\,\diff t$,
we have
\[
 |L_A(\sigma)|\le C_1\e^{-c\sigma^{q'}}+\frac{C}{\kappa\sigma}\,\e^{-\kappa A\sigma}.
\]
Since $q'>1$, it follows that
\[
 \limsup_{\sigma\to+\infty}\sigma^{-1}\log|L_A(\sigma)|
 \le-\kappa A.
\]
If $f$ is nonzero, let $a=\inf\operatorname{ess\,supp}f$.
For every $A>a$, Lemma \ref{lem:superlinear-support} gives
$\limsup\sigma^{-1}\log|L_A(\sigma)|=-a$.
Thus $a\ge\kappa A$ for every $A>a$, a contradiction as $A\to+\infty$.
\end{proof}

\begin{proof}[Proof of Theorem \ref{thm:superlinear-uniqueness}]
Put $f(x)=u(x,0)$. The assumed growth bound gives
$|f(x)|\le C\exp(\theta_1x^q)$. For $\sigma>0$ define
\[
 V_\sigma(x)=\int_0^{+\infty}\e^{-t/\sigma}u(x,-t)\,\diff t.
\]
Since $q''<1$, this integral converges absolutely for every $x\ge0$.
By Lemma \ref{lem:superlinear-volterra}, integration term by term gives
\[
 V_\sigma(x)=\sigma f(x)+\sigma^2 e^{\sigma x}\int_0^xe^{-\sigma t}f(t)\,\diff t.
\]
This yields the exact resolvent identity
\begin{equation}\label{eq:superlinear-resolvent}
 \int_0^x e^{-\sigma t}f(t)\,\diff t
 =\sigma^{-2}\e^{-\sigma x}V_\sigma(x)
 -\sigma^{-1}\e^{-\sigma x}f(x).
\end{equation}

The growth bound gives
\[
 |V_\sigma(x)|\le C \e^{\theta_1x^q}
 \int_0^{+\infty}\exp\left(\theta_2t^{q''}-\frac t\sigma\right)\,\diff t.
\]
A direct maximization, using $q'=q/(q-1)$, gives
\[
 \max_{t\ge0}\left(\theta_2t^{q''}-\frac t\sigma\right)
 =D_q\theta_2^{(2q-1)/(q-1)}\sigma^{q'},
 \qquad
 D_q=\frac{q-1}{2q-1}
 \left(\frac q{2q-1}\right)^{q/(q-1)}.
\]
Standard one-dimensional Laplace estimates therefore give
\[
\int_0^{+\infty}\exp\left(\theta_2t^{q''}-\frac t\sigma\right)\,\diff t
\le\exp\Big(
\big(D_q\theta_2^{(2q-1)/(q-1)}+\ordo(1)\big)\sigma^{q'}
\Big).
\]
Choose $\alpha=(q\theta_1)^{-1/(q-1)}$ and
$x=\alpha\sigma^{q'-1}$. Since
\[
 \theta_1\alpha^q-\alpha
 =-\frac{q-1}{q}(q\theta_1)^{-1/(q-1)},
\]
our parameter condition in Theorem \ref{thm:superlinear-uniqueness} is equivalent to
\[
 D_q\theta_2^{(2q-1)/(q-1)}
 <\frac{q-1}{q}(q\theta_1)^{-1/(q-1)}.
\]
It follows from \eqref{eq:superlinear-resolvent} that for some $c>0$,
\[
 \int_0^{\alpha\sigma^{q'-1}}\e^{-\sigma t}f(t)\,\diff t
 =\Ordo(\e^{-c\sigma^{q'}}).
\]
Finally $\theta_1\alpha^{q-1}=1/q<1$, so Lemma
\ref{lem:superlinear-moving} gives $f=0$. The Volterra representation
then implies $u=0$ on the whole quarter-plane.
\end{proof}

\begin{proof}[Proof of Theorem \ref{thm:superlinear-sharpness}]
Choose $0<a<b$ and a nonzero continuous function $f$ supported
in $[a,b]$. By Lemma \ref{lem:superlinear-volterra}, the solution
with $u(x,0)=f(x)$ and $u(0,y)=0$ is
\[
 u(x,y)=f(x)+\int_0^x
 \sqrt{\frac{|y|}{x-t}}\,
 I_1\!\left(2\sqrt{|y|(x-t)}\right)f(t)\,\diff t.
\]
The kernel is nonnegative for $y\le0$. Since $f$ vanishes for $t<a$,
we have, with $r=|y|$ and $X=(x-a)_+$,
\begin{align*}
 |u(x,y)|
 &\le \|f\|_\infty\left(1+
 \int_0^X\sqrt{\frac r s}\,I_1(2\sqrt{rs})\,\diff s\right)\\
 &=\|f\|_\infty I_0(2\sqrt{rX})
 \le\|f\|_\infty\exp(2\sqrt{rX}).
\end{align*}
Here the integral identity follows by differentiating
$I_0(2\sqrt{rs})$ with respect to $s$, and the last inequality is
standard for the modified Bessel function $I_0$.

The weighted Young inequality
\[
 2\sqrt{Xr}\le\theta_1X^q+\theta_2r^{q''},
 \qquad X,r\ge0,
\]
holds precisely when
\[
 \theta_1\theta_2^{2q-1}\ge
 \frac{(2q-1)^{2q-1}}{q^{2q}},
\]
which is equivalent to the assumed parameter inequality.
Since $X\le x$, it follows that
\[
 |u(x,y)|\le\|f\|_\infty
 \exp\big(\theta_1x^q+\theta_2|y|^{q''}\big).
\]
Finally $u(x,0)=f(x)$ is not identically zero. This proves
sharpness, including the equality case.
\end{proof}

\section{Uniqueness in the critical growth rate case $q=\frac12$}
\label{sec:critical-1}

\subsection{The setting of critical growth $q=\frac12$}
As our starting point, we have a solution to the Klein-Gordon equation 
$u''_{xy}+u=0$ in the spacelike quarter-plane $\R_{>0}\times\R_{<0}$ which
extends continuously to the closed quarter-plane $\R_{\ge0}\times\R_{\le0}$,
with vanishing vertical boundary values:
\[
u(0,y)=0,\qquad y\in\R_{\le0}.
\]
As explained in the proof of Theorem \ref{thm:main-1-0}, $u=\Rop[f,0]$ if the
continuous function $f$ is given by $f(x)=u(x,0)$. We consider $q=\frac{1}{2}$,
and suppose that the growth of $u$ is controlled by
\begin{equation}
|u(x,y)|\le C_0 \exp\big(\theta_1 x^{\frac12}+\beta(y)),\qquad
(x,y)\in\R_{\ge0}\times\R_{\le0},
\label{eq:beta-1}
\end{equation}
where $\beta(y)$ is the function 
\[
\beta(y)=|y|\exp(\theta_2\sqrt{|y|}),
\]
which is even and is strictly increasing and strictly convex on $\R_{\ge0}$, by inspection 
of the first and second derivatives.

\begin{proof}[Proof of Theorem \ref{thm:main-3}]
Since we easily establish that
\[
\exp(\theta_2\sqrt{|y|})\le \e^{\theta_2}+|y|\exp(\theta_2\sqrt{|y|})=\beta(y)+\Ordo(1),
\qquad y\in\R,
\]
the assumption of the theorem entails that we are in the setting of the inequality
\eqref{eq:beta-1}.
We then apply the Fourier-Laplace transform with respect to $x$:
\[
\Lop[u(\cdot,y)](\zeta)=\int_0^{+\infty}\e^{-x\zeta}u(x,y)\,\diff x,\qquad 
\zeta\in\C_{\Re>0},
\]
which defines a holomorphic function in the given half-plane.
Moreover, in view of
the bound \eqref{eq:beta-1}, we have 
\begin{equation}
|\Lop[u(\cdot,y)](\zeta)|\le C_0 \exp(\beta(y))\int_0^{+\infty}
\e^{\theta_1 x^{\frac12}-x\Re\zeta}\,\diff x,\qquad
(\zeta,y)\in\C_{\Re>0}\times\\R_{\le0},
\label{eq:beta-2}
\end{equation}
and since 
\[
\Lop[u(\cdot,y)](\zeta)=\e^{-y/\zeta}\Lop[f](\zeta),
\]
we derive that
\begin{equation}
|\Lop[f](\zeta)|\le C_0 \exp\big(\beta(y)+y\Re(1/\zeta)\big)\int_0^{+\infty}
\e^{\theta_1 x^{\frac12}-x\Re\zeta}\,\diff x,\qquad
\zeta\in\C_{\Re>0},
\label{eq:beta-3}
\end{equation}
where $y\in\R_{\le0}$ is arbitrary. If we write
\[
\beta^\ast(t):=\sup\{yt-\beta(y):\,\,\,y\in\R_{\ge0}\}\ge0,
\]
the result is another convex function that is, in the sense of Legendre,
the convex conjugate to $\beta$ itself. By performing asymptotic calculations, 
we can establish that
\begin{equation}
\beta^\ast(t)=\frac{1}{\theta_2^2}\,t\log^2 t-\frac{2}{\theta_2^2}\,t\log t\log\log t
+\Ordo(t\log t)
\label{eq:beta-3.1}
\end{equation}
holds asymptotically as $t\to+\infty$.
It now follows from \eqref{eq:beta-3} that
\begin{equation}
|\Lop[f](\zeta)|\le C_0 \exp\big(-\beta^\ast(\Re(1/\zeta))\big)\int_0^{+\infty}
\e^{\theta_1 x^{\frac12}-x\Re\zeta}\,\diff x,\qquad
\zeta\in\C_{\Re>0},
\label{eq:beta-4}
\end{equation}
and if we use Lemma \ref{lem:basic-est-1} to estimate the integral, we find that
\[
\int_0^{+\infty}\e^{\theta_1 x^{\frac12}-x\sigma}\,\diff x\le
\frac{1}{\sigma}\big(2^{3/4}\sqrt{\pi}\theta_1\sigma^{-\frac12}+8\big)
\exp\bigg(\frac{\theta_1^2}{4\sigma}\bigg),
\]
so that if we write $\Re\zeta=\sigma$, then
\begin{equation}
|\Lop[f](\zeta)|\le C_0 
\frac{1}{\sigma}\big(2^{3/4}\sqrt{\pi}\theta_1\sigma^{-\frac12}+8\big)
\exp\bigg(-\beta^\ast(\Re(1/\zeta))+
\frac{\theta_1^2}{4\sigma}\bigg),
\label{eq:beta-5}
\end{equation}
for $\Re\zeta=\sigma>0$. Since the point we are interested in is actually the
origin $\zeta=0$, which after inversion becomes the point at infinity, it makes
sense to consider the function
\[
F(\zeta):=\Lop[f](1/\zeta),\qquad\zeta\in\C_{>0},
\]
for which the estimate \eqref{eq:beta-5} becomes
\begin{equation}
|F(\zeta)|\le C_0 
\frac{|\zeta|^2}{\sigma}\big(2^{3/4}\sqrt{\pi}\theta_1|\zeta|\sigma^{-\frac12}+8\big)
\exp\bigg(-\beta^\ast(\sigma)+
\frac{\theta_1^2|\zeta|^2}{4\sigma}\bigg),\qquad\zeta\in\C_{\Re>0}.
\label{eq:beta-6}
\end{equation}
If we consider the region where
\[
\Omega_1:=\bigg\{\zeta\in\C_{\Re>0}:\,\,\,
\sigma\beta^\ast(\sigma)>\frac{\theta_1^2|\zeta|^2}{4}
\bigg\},
\]
then
\begin{equation}
|F(\zeta)|\le C_0 
\frac{|\zeta|^2}{\sigma}\big(2^{3/4}\sqrt{\pi}\theta_1|\zeta|\sigma^{-\frac12}+8\big),
\qquad\zeta\in\Omega_1,
\label{eq:beta-7}
\end{equation}
while 
\begin{equation}
|F(\sigma)|\le C_0 
\big(2^{3/4}\sqrt{\pi}\theta_1\sigma^{\frac32}+8\sigma\big)
\exp\bigg(-\beta^\ast(\sigma)+
\frac{\theta_1^2\sigma}{4}\bigg),\qquad 0<\sigma<+\infty.
\label{eq:beta-8}
\end{equation}
We factor $\beta^\ast(t)=t\alpha^\ast(t)$, and see that the condition
defining the set $\Omega_1$ is
\[
\frac{|\zeta|^2}{\sigma^2}<\frac{4}{\theta_1^2}\alpha^\ast(\sigma),\qquad 
\zeta\in\Omega_1,
\]
that is,
\[
\frac{\sigma}{|\zeta|}>\frac{\theta_1}{2}\,\alpha^\ast(\sigma)^{-\frac12},
\qquad \zeta\in\Omega_1.
\]
In the complement, we have instead
\[
\frac{\sigma}{|\zeta|}\le\frac{\theta_1}{2}\,\alpha^\ast(\sigma)^{-\frac12},
\qquad \zeta\in\C_{\Re>0}\setminus\Omega_1.
\]
Moreover, in view of \eqref{eq:beta-3.1}, we see that
\begin{equation}
\alpha^\ast(\sigma)^{-\frac12}=\frac{\theta_2}{\log\sigma}+
\theta_2\frac{\log\log \sigma}{\log^2\sigma}+
\Ordo\bigg(\frac{1}{\log^2\sigma}\bigg)
\label{eq:beta-9}
\end{equation}
as $\sigma\to+\infty$. Along the boundary $\partial\Omega_1$, we have equality,
\[
\frac{\sigma}{|\zeta|}=\frac{\theta_1}{2}\,\alpha^\ast(\sigma)^{-\frac12}
=\frac{\theta_1\theta_2}{2}\bigg(\frac{1}{\log\sigma}+\frac{\log\log\sigma}{\log^2\sigma}
+\Ordo(1)\bigg),
\qquad \zeta\in\partial\Omega_1,
\]
where we used the asymptotics in \eqref{eq:beta-9}, so that
\[
\rho:=|\zeta|=\frac{2}{\theta_1}\sigma\alpha^\ast(\sigma)^{\frac12}=\frac{2\sigma}
{\theta_1\theta_2}\big(\log\sigma+\log\log\sigma+\Ordo(1)\big),
\]
as $\sigma\to+\infty$, 
which relationship between $\sigma$ and $\rho$ we may invert, asymptotically:
\[
\sigma=\frac{\theta_1\theta_2}{2}\frac{\rho}{\log\rho}
\bigg(1+\Ordo\bigg(\frac{1}{\log\rho}\bigg)\bigg),
\]
so that
\[
\frac{\sigma}{\rho}=\frac{\theta_1\theta_2}{2\log\rho}\,\big(1+\ordo(1)\big),
\]
as $\rho\to+\infty$.
This means that the symmetric opening angle of $\Omega_1$ on the radius 
$|\zeta|=\rho$ is $\pi-2\epsilon$, where 
\begin{equation}
\sin\epsilon=\frac{\theta_1\theta_2}{2\log\rho}\big(1+\ordo(1)\big),
\label{eq:beta-10}
\end{equation}
and the asymptotics is as $\rho\to+\infty$. Let us consider a slightly smaller domain
$\Omega_2\subset \Omega_1$ which is simply connected and starlike, given as
\[
\Omega_2=\{r\,\e^{\imag\varphi}:\,\,\,\,|\varphi|<\tfrac12\pi-\epsilon(r)\},
\]
where $\epsilon:\,\R_{\ge0}\to[0,\frac12\pi]$ is a decreasing function with 
$\epsilon(r)=\frac{1}{2}\pi$ for $0\le r\le r_0$ with some $r_0>0$, 
and the asymptotic decay rate
\[
\epsilon(r)=\frac{\theta_3}{2\log r},\qquad R\le r<+\infty,\qquad 
\theta_1\theta_2<\theta_3<2\pi.
\]
Note that this agrees with the asymptotic angle dependence on 
the radius in equation \eqref{eq:beta-10}, and the containment 
$\Omega_2\subset\Omega_1$.
Next, suppose $P_\infty$ is a positive harmonic function
in $\Omega_2$ which vanishes along the boundary $\partial\Omega_2$.
Projectively, there is only one such function, which we understand as meaning 
that any other such function is the positive constant multiple of the first function
we found. We may think of $P_\infty$ as the Poisson kernel for the point at infinity,
considered as a boundary point of $\Omega_2$. The classical Ahlfors-Warschawski
theorem \cite{Warsch} characterizes geometrically the growth of the function $P_\infty$ 
along the positive real line. Indeed,
\[
P_\infty(\sigma)\asymp \exp\bigg(\int_{r_1}^{\rho}\frac{\pi}{\pi-2\epsilon(r)}\,
\frac{\diff r}{r}\bigg)
\]
in the asymptotic sense as $\sigma\to+\infty$, and where the radius $r_1$ is selected 
such that, e.g., $\epsilon(r_1)=\frac{1}{4}\pi$.
We then calculate that if $R\ge r_1$ is fixed, we obtain
\begin{multline}
\int_R^{\rho}\frac{\pi}{\pi-2\epsilon(r)}\frac{\diff r}{r}=
\int_{R}^{\rho}\bigg(1+\frac{2}{\pi}\epsilon(r)+\Ordo(\epsilon(r)^2)\bigg)\frac{\diff r}{r}
\\
=\int_{R}^{\rho}\bigg(1+\frac{1}{\pi}\frac{\theta_3}{\log r}+\Ordo((\log r)^{-2})\bigg)
\frac{\diff r}{r}
\\
=\log\rho-\log R+\frac{\theta_3}{\pi}\big(\log\log\rho-\log\log R\big)+\Ordo(1),
\end{multline}
as $\rho\to+\infty$. Since $R$ is kept fixed, we may replace $r_1$ by $R$ in the
Ahlfors-Warschawski theorem, so that
\begin{equation}
P_\infty(\sigma)\asymp\exp\bigg(\int_{R}^{\rho}\frac{\pi}{\pi-2\epsilon(r)}\,
\frac{\diff r}{r}\bigg)
\asymp\frac{\rho}{R}\bigg(\frac{\log\rho}{\log R}\bigg)^{\frac{\theta_3}{\pi}}
\label{eq:beta-11}
\end{equation}
as $\rho\to+\infty$. 
We now return to the estimates \eqref{eq:beta-7} and \eqref{eq:beta-8}.
The growth bound \eqref{eq:beta-7} holds throughout $\Omega_2$.
For sufficiently large $\rho=|\zeta|$, the geometry of $\Omega_2$
gives $\sigma\ge c\rho/\log\rho$. Consequently,
\[
 |F(\zeta)|\le C\rho^{3/2}(\log\rho)^{3/2},
 \qquad \zeta\in\Omega_2,\quad\rho\gg1.
\]
We may therefore use the explicit holomorphic, nonvanishing factor
\[
 G(\zeta):=(1+\zeta)^{-4}.
\]
The estimate \eqref{eq:beta-7} also controls $F$ on the bounded
portion of $\Omega_2$, which is bounded away from the origin by
construction. Hence $FG$ is bounded throughout $\Omega_2$.
On the positive real axis, \eqref{eq:beta-8} gives
\begin{equation}
 |F(\sigma)G(\sigma)|\le C(1+\sigma)^{-4}
 (\sigma^{3/2}+\sigma)
 \exp\bigg(-\beta^\ast(\sigma)+\frac{\theta_1^2\sigma}{4}\bigg),
 \qquad \sigma\to+\infty.
 \label{eq:beta-12}
\end{equation}
Since we know that
\[
\beta^\ast(\sigma)=\frac{1}{\theta_2^2}(1+\ordo(1))\,\sigma\log^2\sigma,
\]
it now follows from this in combination with \eqref{eq:beta-11} that for each 
$\lambda\in\R_{>0}$,
\begin{equation}
\big|F(\sigma)G(\sigma)\big|=\Ordo
\big(\exp\big(-\lambda P_\infty(\sigma)\big)\big)  \,\,\,\text{as}\,\,\,
\sigma\to+\infty.
\label{eq:beta-13}
\end{equation}
This is basically (after a suitable conformal mapping) the same as the setting 
of Theorem \ref{thm:NT-1}, and it follows that $FG=0$ and hence $F=0$ 
and equivalently $f=0$ on $\R_\ge0$. Consequently, 
$u=\Rop[f,0]=\Rop[0,0]=0$ on $\R_{\ge0}
\times\R_{\le0}$.
\end{proof}

\subsection{Counterexamples for the critical growth condition}\label{sec:counter-critical}
We now turn to the construction of nontrivial solutions in the critical growth regime, showing that uniqueness may fail when the growth parameters are sufficiently large.

Let $\theta_1,\theta_2>0$. For an auxiliary parameter $\eta>0$, put
\[
 \beta_\eta(y)=|y|\exp\bigl(\eta\sqrt{|y|}\bigr),\qquad
 \beta_\eta^*(s)=\sup_{r\ge0}\{rs-r e^{\eta\sqrt r}\},\quad s\ge0.
\]
We use the asymptotic formula from (6.1.3),
\begin{equation}\label{eq:counter-critical-legendre}
 \beta_\eta^*(s)=\frac{s\log^2s}{\eta^2}+o(s\log^2s),\qquad s\to+\infty.
\end{equation}
The function $\beta_\eta^*$ is continuous and nonnegative on $\R_{\ge0}$.

\begin{proof}[Proof of Theorem \ref{eq:counter-critical-prop:counterexample}]
Choose $\eta$ with $2\pi/\theta_1<\eta<\theta_2$. Since
$\theta_1^2/(4\pi^2)>\eta^{-2}$, we may choose $a>0$ such that
\begin{equation}\label{eq:counter-critical-delta}
 a-\frac{\pi^2a^2}{\theta_1^2}>\eta^{-2}.
\end{equation}
For $w=s+it$ in the right half-plane, write
$\log(1+w)=U+iV$ with the principal logarithm. Then
$U=\log|1+w|\ge\log(1+s)$ and $|V|<\pi/2$. Consequently,
\begin{align*}
 \operatorname{Re}\bigl[w(\log(1+w))^2\bigr]
 &=s(U^2-V^2)-2tUV\\
 &\ge sU^2-\frac{\pi^2s}{4}-\pi|t|U.
\end{align*}
Completing the square yields
\begin{align}
&a\operatorname{Re}\bigl(w(\log(1+w))^2\bigr)
       +\frac{\theta_1^2|w|^2}{4s}
\quad\ge asU^2-\frac{a\pi^2s}{4}-a\pi|t|U
                   +\frac{\theta_1^2s}{4}+\frac{\theta_1^2t^2}{4s}\notag\\
&\quad\ge \Big(a-\frac{\pi^2a^2}{\theta_1^2}\Big)sU^2
                  +\Big(\frac{\theta_1^2}{4}-\frac{a\pi^2}{4}\Big)s
\quad\ge \Big(a-\frac{\pi^2a^2}{\theta_1^2}\Big)s\log^2(1+s)-C_1s.
\label{eq:counter-critical-uniform}
\end{align}
Here the constant $C_1\ge0$ depends only on $a,\theta_1$. In view of
\eqref{eq:counter-critical-legendre} and \eqref{eq:counter-critical-delta},
the expression
\[
\Big(a-\frac{\pi^2a^2}{\theta_1^2}\Big)\,s\log^2(1+s)-C_1s-\beta_\eta^*(s)
\]
is bounded from below on $\R_{\ge0}$. Hence there is a constant $C_2$ such that
\begin{equation}\label{eq:counter-critical-key}
 a\operatorname{Re}\bigl(w(\log(1+w))^2\bigr)
 +\frac{\theta_1^2|w|^2}{4s}\ge\beta_\eta^*(s)-C_2,
 \qquad s=\operatorname{Re}w>0.
\end{equation}
We define, for $\operatorname{Re}\zeta>0$,
\begin{equation}\label{eq:counter-critical-transform}
 F(\zeta)=\exp\bigg(-\frac a\zeta
       \Big(\log\Big(1+\frac1\zeta\Big)\Big)^2\bigg),
 \qquad G(\zeta)=\frac{F(\zeta)}{(1+\zeta)^2}.
\end{equation}
These functions are holomorphic in the right half-plane. Substituting
$w=1/\zeta$ in \eqref{eq:counter-critical-key} gives, with
$\sigma=\operatorname{Re}\zeta$,
\begin{equation}\label{eq:counter-critical-Fbound}
 |F(\zeta)|\le \e^{C_2}\exp\bigg(
 -\beta_\eta^*\Big(\operatorname{Re}\frac1\zeta\Big)
 +\frac{\theta_1^2}{4\sigma}\bigg).
\end{equation}
For $y\le0$, the defining inequality for the Legendre transform asserts that
$|y|r-\beta_\eta^*(r)\le\beta_\eta(y)$ holds for $r\ge0$. Consequently, we have
the estimate
\begin{equation}\label{eq:counter-critical-weighted}
 |\e^{-y/\zeta}F(\zeta)|\le \e^{C_2}
           \exp\left(\beta_\eta(y)+\frac{\theta_1^2}{4\sigma}\right).
\end{equation}
For any $\sigma>0$, define the inverse Laplace integral
\begin{equation}\label{eq:counter-critical-inverse}
 u(x,y)=\frac1{2\pi \imag}\int_{\sigma-\imag\infty}^{\sigma+\imag\infty}
\e^{x\zeta-y/\zeta}G(\zeta)\,\diff\zeta.
\end{equation}
The bound \eqref{eq:counter-critical-weighted} and the factor $(1+\zeta)^{-2}$
make the integral absolutely convergent. Cauchy's theorem, using
horizontal segments and the same uniform bounds on closed vertical
strips, shows that it is independent of $\sigma>0$.
Furthermore,
\[
 \frac1{2\pi}\int_{-\infty}^{+\infty}
 \frac{\diff t}{|1+\sigma+\imag t|^2}=\frac1{2(1+\sigma)}.
\]
It follows that
\begin{equation}\label{eq:counter-critical-sigmabound}
 |u(x,y)|\le\frac{\e^{C_2}}{2(1+\sigma)}
 \exp\left(\sigma x+\frac{\theta_1^2}{4\sigma}+\beta_\eta(y)\right).
\end{equation}
For $x>0$, choose $\sigma=\frac{\theta_1}{2\sqrt{x}}$. Then
$\sigma x+\frac{\theta_1^2}{4\sigma}=\theta_1\sqrt x$,
which gives the auxiliary estimate
\[
 |u(x,y)|\le C\exp\bigl(\theta_1\sqrt{x}+\beta_\eta(y)\bigr).
\]
Since $\eta<\theta_2$, the function
$r \e^{\eta\sqrt r}-\e^{\theta_2\sqrt r}$ is bounded above
for $r\ge0$. Hence there is $M<+\infty$ such that
$\beta_\eta(y)\le \e^{\theta_2\sqrt{|y|}}+M$ for all $y\le0$.
Absorbing $e^M$ into the multiplicative constant proves the asserted
growth bound. At $x=0$, letting
$\sigma\to+\infty$ in \eqref{eq:counter-critical-sigmabound} gives $u(0,y)=0$.
Continuity up to $x=0$ follows from \eqref{eq:counter-critical-sigmabound}
uniformly for $y$ in compact subsets of $\R_{\le0}$.

An elementary argument verifies that $u$ solves the Klein-Gordon equation
$u''_{xy}+u=0$ in the sense of distribution theory on the open quarter-plane
$\R_{>0}\times\R_{<0}$. It is also elementary to check that $u$ is continuus in
the closed quarter-plane. 
Finally, $u$ cannot vanish identically: its boundary trace at $y=0$
has Laplace transform $G$, and $G\not\equiv0$.
\end{proof}

\begin{proof}[Proof of Theorem \ref{prop:counter-critical-equality}]
Since $\theta_1\theta_2=2\pi$, we have $\theta_1^2/4=\pi^2/\theta_2^2$.
This time around, we use the original weight
$\gamma(r)=\exp(\theta_2\sqrt r)$, $r\ge0$, and its Legendre conjugate
$\gamma^*(s)=\sup_{r\ge0}\{sr-\gamma(r)\}$.
As $s\to+\infty$, let $r_*(s)$ denote the large maximizing
stationary point in the definition of $\gamma^*(s)$. Then
\begin{equation}\label{eq:critical-gamma-star}
 \gamma^*(s)=\frac{s}{\theta_2^2}
 \bigl((\log s)^2+2\log s\log\log s+O(\log s)\bigr).
\end{equation}
Indeed, the stationary-point equation gives
\[
s=\frac{\theta_2}{2\sqrt{r_*}}\exp(\theta_2\sqrt{r_*}),
\qquad
\theta_2\sqrt{r_*}=\log s+\log\log s+\mathrm{O}(1).
\]
At this point,
\[
\gamma^*(s)=sr_*-\frac{2s\sqrt{r_*}}{\theta_2},
\]
which yields the asserted asymptotics.
For $\Re w>0$, put
\[
 h(w)=\log(1+w)+2\log\log(\e+w),\qquad
 \Phi(w)=\frac{2}{\theta_2^2}\,w\, h(w)^2.
\]
Here, the logarithms are the standard principal branch logarithm,
and the expressions define holomorphic functions here.
Writing $h=h_1+\imag h_2$, with $h_1,h_2$ real-valued, we have
\begin{equation}\label{eq:critical-U-bound}
 h_1(w)\ge \log(1+\re w)+2\log\log(\e+\re w)\ge0.
\end{equation}
We also have the uniform estimate
\begin{equation}\label{eq:critical-V-bound}
 h_2(w)^2\le\frac{\pi^2}{4}+\mathrm{O}\bigl((1+h_1(w))^{-1}\bigr).
\end{equation}
Indeed, $\re\log(\e+w)=\log|e+w|\ge1$ and
\[
 \bigl|\arg\log(\e+w)\bigr|
 \le \frac{|\arg(\e+w)|}{\log|e+w|}
 \le \frac{\pi}{2\log|\e+w|}.
\]
Moreover,
\[
 1+h_1(w)\le \mathrm{O}\bigl(\log|\e+w|\bigr),
\]
since $\log|1+w|\le\log|e+w|$ and
$|\log(\e+w)|\le\log|e+w|+\pi/2$.
Together with $|\arg(1+w)|<\pi/2$, these estimates give
\eqref{eq:critical-V-bound}.
Completing the square then gives
\begin{align*}
 \re\,\Phi(w)+\frac{\theta_1^2|w|^2}{4\re\, w}
 &=\frac{\re\, w}{\theta_2^2}
 \bigl(2(h_1^2-h_2^2)+\pi^2\bigr)
 -\frac{4(\im\, w)h_1h_2}{\theta_2^2}
 +\frac{\pi^2(\im\, w)^2}{\theta_2^2\re\, w}\\
 &\ge\frac{\re\, w}{\theta_2^2}
 \biggl(2h_1^2-2h_2^2+\pi^2
 -\frac{4h_1^2h_2^2}{\pi^2}\biggr)\ge\frac{\re\, w}{\theta_2^2}
 \bigl(h_1^2-\mathrm{O}(1+h_1)\bigr).
\end{align*}
By \eqref{eq:critical-U-bound}, for large $\re\, w$ this is at least
\[
 \frac{\re\, w}{\theta_2^2}\bigl((\log\re\, w)^2
 +4\log(\re\, w)\log\log(\re\, w)-\mathrm{O}(\log\re\, w)\bigr).
\]
Comparison with \eqref{eq:critical-gamma-star}
shows that, for some constant $C$, uniformly for $\re\, w>0$,
\begin{equation}\label{eq:critical-equality-global}
 \Re\Phi(w)+\frac{\theta_1^2|w|^2}{4\re\, w}\ge\gamma^*(\re\, w)-C.
\end{equation}
After enlarging the constant $C$, careful examination of the expressions
involved reveals that the same inequality is uniform in $\im\, w$ as well.

For $\re\,\zeta>0$, we put
\[
F(\zeta)=\exp\big(-\Phi(1/\zeta)\big),\qquad
G(\zeta)=\frac{F(\zeta)}{(1+\zeta)^2},
\]
and define $u$ by the inverse Laplace integral
\eqref{eq:counter-critical-inverse} with this choice of $G$.
For $y\le0$ and $\re\,w>0$, Fenchel's inequality gives
\[
 |y|\,\re\,w-\gamma(|y|)\le\gamma^*(\re\,w).
\]
Combining this with \eqref{eq:critical-equality-global}, we obtain
\[
\re\,\Phi(w)+\frac{\theta_1^2|w|^2}{4\re\,w}
\ge |y|\,\re\,w-\gamma(|y|)-C.
\]
Now put $w=1/\zeta$. Since
\[
 \frac{|w|^2}{\re\,w}=\frac1{\re\,\zeta},
\]
we conclude that
\[
 |\e^{-y/\zeta}F(\zeta)|
 \le \e^C\exp\left(\frac{\theta_1^2}{4\Re\zeta}+\gamma(|y|)\right),
 \qquad y\in\R_{\le0}.
\]
The same absolute-convergence and contour-shifting argument as in the
preceding proposition yields, for every $\sigma>0$,
\[
 |u(x,y)|\le\frac{\e^C}{2(1+\sigma)}
 \exp\left(\sigma x+\frac{\theta_1^2}{4\sigma}+\exp(\theta_2\sqrt{|y|})\right).
\]
Choosing $\sigma=\frac{\theta_1}{2\sqrt{x}}$ for $x>0$ proves the
required growth bound, since the minimum of
$\sigma x+\theta_1^2/(4\sigma)$ is $\theta_1\sqrt{x}$.
Sending $\sigma\to+\infty$ at $x=0$ gives $u(0,y)=0$.
The equation in the sense of distributions and continuity follow
from the same truncated-inversion argument as above. Nontriviality
follows from uniqueness of the Laplace transform, since
$G\not\equiv0$.
\end{proof}

\begin{rem}
The auxiliary weight $\beta_\eta$ suffices when
$\theta_1\theta_2>2\pi$, but it loses the secondary term required
at equality.  The corrected holomorphic phase above works directly
with the original weight $\gamma(r)=\e^{\theta_2\sqrt r}$.
Thus the borderline $\theta_1\theta_2=2\pi$ is covered by a
separate construction.
\end{rem}

\section{Appendix: Growth bounds for stable solutions}
\label{sec:appendix-1}

\subsection*{Proof of the stable-solution growth bounds}

We first supply the elementary estimate that justifies the bounds stated in
Proposition~\ref{prop:stable-growth}.

\begin{proof}[Proof of Proposition~\ref{prop:stable-growth}]
We give the argument for $\Sop_-[f]$; the other assertion follows by
interchanging $x$ and $y$.  Put
\[
 N_\epsilon(f):=\sup_{t\in\R}(1+|t|)^{\frac14+\epsilon}|f(t)|,
 \qquad \delta:=\min\{\epsilon,\tfrac12\}.
\]
Here and below the constants are absolute for $0<\delta\le\frac12$,
except for their dependence on $N_\epsilon(f)$.
The standard bounds $|J_1(r)|\le C\min\{r,r^{-1/2}\}$ for $r>0$
give, whenever $sv>0$,
\[
 |J_{0,1}(s,v)|
 =\left|\left(\frac vs\right)^{1/2}J_1(2\sqrt{sv})\right|
 \le C\min\{|v|,|v|^{1/4}|s|^{-3/4}\}
 \le C|v|^{1/4}|s|^{-3/4}.
\]
Consequently, in either of the two half-planes $y\ge0$ and $y<0$,
\[
 |\Sop_-[f](x,y)-f(x)|
 \le C N_\epsilon(f)|y|^{1/4}
 \int_0^{+\infty} \frac{\diff s}{s^{3/4}(1+|x\pm s|)^{1/4+\delta}}.
\]
Here the minus sign corresponds to $y\ge0$, and the plus sign to $y<0$.
The integral is bounded by $C/\delta$, uniformly in $x\in\R$.
Indeed, for $x\le0$ with the minus sign (or $x\ge0$ with the plus
sign), the integrand is majorized by
$s^{-3/4}(1+s)^{-1/4-\delta}$, whose integral is $O(\delta^{-1})$.
For the remaining cases, by symmetry it suffices to take $x\ge0$
and the minus sign.  When $x\le2$, the assertion is immediate by
splitting the integral at $s=4$.  For $x>2$, split the integral into
$0<s<x/2$, $x/2\le s\le2x$, and $s>2x$.  The first and last
pieces are $O(\delta^{-1})$ by direct integration.  For the middle
piece we use $s^{-3/4}\le C\,x^{-3/4}$ and
\[
 \int_{x/2}^{2x}(1+|x-s|)^{-1/4-\delta}\,\diff s
 \le Cx^{3/4},
\]
which proves the required uniform estimate.  Since
$\delta^{-1}\le 2\max\{1,\epsilon^{-1}\}$, we obtain
\[
 |\Sop_-[f](x,y)|\le C N_\epsilon(f)
 \left((1+|x|)^{-1/4-\epsilon}
       +\max\{1,\epsilon^{-1}\}|y|^{1/4}\right).
\]
In particular, for $0<\epsilon\le1$ this is the asserted bound,
with a constant independent of $\epsilon$ after normalization by
$N_\epsilon(f)$.  The same kernel estimate gives absolute and
locally uniform convergence of the defining integrals.  The
Klein-Gordon equation and the prescribed trace follow by the
standard argument involving differentiation in the sense of distributions for the Riemann
kernel (or by approximation of $f$ by smooth compactly supported
data).

For uniqueness, the difference of two solutions with the indicated
trace and growth bound is a continuous solution vanishing on the
$x$-axis and having at most polynomial growth in the whole plane.
After interchanging the characteristic variables, the Liouville
theorem of the preceding subsection applies in the appropriate
spacelike quarter-plane, since a polynomial bound is dominated by
every admissible subcritical exponential bound.  This proves
uniqueness there, and hence in all four quarter-planes.  The
argument for $\Sop_\vert[g]$ is identical, with the variables
interchanged, as permitted by Lorentz invariance.
\end{proof}

\subsection*{Acknowledgements}
The author acknowledges support from grant 075-15-2025-013
of the Government of the Russian Federation for the state support of
scientific research, carried out under the supervision of leading
scientists.

\subsection*{Conflict of interest and data availability}
There is no conflict of interest pertaining to this work.
Moreover, no data is used.

\end{document}